\documentclass[11pt]{amsart}
\usepackage[
text={450pt,575pt},
headheight=9pt,
centering
]{geometry}

\usepackage{pictexwd,dcpic}
\usepackage{latexsym}
\usepackage[all,cmtip]{xy}
\usepackage{times}
\usepackage{braket}
\usepackage[english]{babel}
\usepackage[utf8]{inputenc}
\usepackage{faktor}
\usepackage{paralist,url,verbatim}
\usepackage{amscd}
\usepackage[T1]{fontenc}
\usepackage{xspace}
\usepackage{comment}
\usepackage{amssymb}
\usepackage{amsmath,amsthm,amssymb,amsfonts}
\usepackage{mathtools}
\usepackage{enumitem}
\usepackage{wasysym}
\usepackage[usenames,dvipsnames]{xcolor}
\usepackage[bookmarksopen=true]{hyperref}
\definecolor{darkblue}{RGB}{0,0,160}
\definecolor{darkgreen}{RGB}{0,80,0}
\hypersetup{
colorlinks,%
citecolor=black,%
filecolor=black,%
linkcolor=darkblue,%
urlcolor=darkblue,%
}

\newcommand{\cT}{\mathcal{T}}

\newcommand{\ktrash}[1]{}

\newcommand{\change}[1]{}
\definecolor{cKlaus}{rgb}{0.15,0.40,0.03}
\definecolor{cKLinkGBU}{rgb}{1,0,0}  
\definecolor{cKLink}{rgb}{0.6,0.2,0.3}
\definecolor{cALink}{rgb}{0,0.3,0}
\usepackage[textsize=tiny]{todonotes}

\usepackage{tikz}
\usetikzlibrary{decorations.markings,intersections,positioning,calc}

  \tikzset{mylabel/.style  args={at #1 #2  with #3}{
    postaction={decorate,
    decoration={
      markings,
      mark= at position #1
      with  \node [#2] {#3};
 } } } }

\theoremstyle{plain}
\newtheorem{theorem}{Theorem}[section]

\newtheorem{corollary}[theorem]{Corollary}
\newtheorem{lemma}[theorem]{Lemma}
\newtheorem{proposition}[theorem]{Proposition}

\theoremstyle{definition}

\newtheorem{remark}[theorem]{Remark}
\newtheorem{definition}[theorem]{Definition}

\newtheorem{example}[theorem]{Example}

\newcommand{\Hom}{\operatorname{Hom}}

\newcommand{\NN}{\ensuremath{\mathbb{N}}}

\newcommand{\QQ}{\ensuremath{\mathbb{Q}}}

\newcommand{\CC}{\ensuremath{\mathbb{C}}}

\newcommand{\ZZ}{\ensuremath{\mathbb{Z}}}

\newcommand{\paar}{(X,\partial X)}
\newcommand{\bfa}{\ensuremath{\mathbf{a}}}

\newcommand{\bfk}{\ensuremath{\mathbf{k}}}

\newcommand{\bfx}{\ensuremath{\mathbf{x}}}

\newcommand{\bfZ}{\ensuremath{\mathbf{Z}}}

\newcommand{\conv}{\operatorname{conv}}

\newcommand{\kenditem}{\vspace{-1ex}\end{itemize}}

\newcommand{\kitem}{\begin{itemize}\vspace{-2ex}}

\newcommand{\tM}{\widetilde{M}}
\newcommand{\tN}{\widetilde{N}}

\newcommand{\wt}[1]{\widetilde{#1}}

\newcommand{\Der}{\operatorname{Der}}

\newcommand{\kS}{{S}} 
\newcommand{\bound}{{\partial}} 
\newcommand{\ktT}{{\wt{T}}} %
\newcommand{\ktS}{{\wt{S}}} 
\newcommand{\kMT}{{\mathcal T}} 
\newcommand{\kMTZ}{{\mathcal T_{\ZZ}}} 

\newcommand{\demo}{\hfill $\triangle$}

\newcommand{\spec}{\operatorname{Spec}}

\newcommand{\un}{\underline}
\renewcommand{\span}{\operatorname{Span}}

\newcommand{\lan}{\langle}
\newcommand{\ran}{\rangle}

\newcommand{\newt}{\Delta}

\usepackage{tikz-cd}

\newcommand{\cB}{\mathcal{B}}

\newcommand{\cI}{\mathcal{I}}
\newcommand{\cJ}{\mathcal{J}}

\newcommand{\cN}{\mathcal{N}}

\newcommand{\cS}{\mathcal{S}}

\newcommand{\cX}{\mathcal{X}}
\newcommand{\cY}{\mathcal{Y}}

\newcommand{\cR}{\mathcal{R}}
\newcommand{\bfT}{\bold{T}}

\newcommand{\bo}{\partial}

        \renewcommand{\bo}{\bound}

     \newcommand{\bfu}{\ensuremath{\mathbf{u}}}

\newcommand{\eps}{\epsilon}
\newcommand{\lam}{\lambda}

\newcommand{\tT}{\tilde{T}}

\newcommand{\te}{\tilde{\eta}}

\begin{document}

\title{Deformations of an affine Gorenstein toric pair}
\author{Matej Filip}
\address{University of Ljubljana, Institute of Mathematics, Physics and Mechanics, Trzaska cesta 25, Ljubljana, Slovenia}
\email{matej.filip@fe.uni-lj.si}
\thanks{Supported by Slovenian Research Agency program P1-0222 and grant J1-60011}

\begin{abstract}
We consider deformations of a pair $(X,\partial X)$, where $X$ is an affine toric Gorenstein variety and $\partial X$ is its boundary. We compute the tangent and obstruction space for the corresponding deformation functor and for an admissible  lattice degree $m$ we construct the miniversal deformation of $(X,\partial X)$ in degrees $-km$, for all $k\in \NN$. This in particular generalizes Altmann's construction of the miniversal deformation of an isolated Gorenstein toric singularity   to an arbitrary non-isolated Gorenstein toric singularity. Moreover, we show that the irreducible components of the reduced miniversal deformation are in one to one correspondence with maximal Minkowski decompositions of the polytope $P\cap (m=1)$, where $P$ is the lattice polytope defining $X$.
\end{abstract}



\maketitle

\section{Introduction}

Mirror symmetry suggests that there is a relationship between Fano manifolds and certain Laurent polynomials, cf.\ \cite{cor}. 
More precisely, if a Laurent polynomial $f$ is mirror to a Fano manifold $Y$, it is expected that a Fano manifold $Y$ admits a $\QQ$-Gorenstein degeneration to a singular toric variety, whose fan is the spanning fan of the Newton polytope $\newt(f)$. 
Thus studying deformations of toric varieties is important for classifying  Fano varieties, which has become an extensive research area in recent years.

Altmann \cite{alt} constructed the miniversal deformation space of an isolated affine Gorenstein toric singularity and in \cite{kas}, \cite{m} and \cite{a} his work was generalized to some other cases that will be described more in detail below. 
The deformations of the boundary $\partial X$ of a toric variety (or more generally the deformations of toroidal crossing spaces, which are locally isomorphic to boundary divisors of toric varieties) are studied using log geometry and appear also in the Gross-Siebert program, cf.\ \cite{gs}, \cite{ffr}.

Mirror symmetry also suggests to work with deformations of $\paar$, which means deformations of a closed embedding $\partial X\hookrightarrow X$, instead of only working  with deformations of $X$ or $\partial X$, see e.g.\ \cite{mojcor}, where Corti, Petracci and the author analyse the case when $X$ is a three-dimensional affine Gorenstein toric variety (having possibly non-isolated singularities).
In this case we state a conjecture that there is a canonical bijective correspondence between the set of smoothing components of $X$ (which is the same as the set of smoothing components of $\paar$) and the set of certain Laurent polynomials.

It is well known that Gorenstein toric Fano varieties are in correspondence with reflexive polytopes. By a comparison theorem of Kleppe \cite{kle}, the deformations of Gorenstein toric Fano varieties can be  obtained by (degree $0$) deformations of their affine cones, which are affine Gorenstein toric varieties. Some of the  deformations of an affine toric variety that we construct in this paper also induce deformations of the corresponding Gorenstein toric Fano varieties as we will see below.

Let $P \subset N_\QQ$ be a lattice polytope in a finite-dimensional lattice $N$. From now on, let $X$ be an affine Gorenstein toric variety, given by a rational polyhedral cone $\sigma \subset \widetilde{N}_\QQ$, where $\widetilde{N} := N \oplus \ZZ$, and $\sigma$ is the cone over $P$, embedded in the hyperplane $N_\QQ \times {1} \subset \widetilde{N}_\QQ$. There exists $R^*\in \widetilde{M}:=\Hom(\widetilde{N},\ZZ)$ such that the integral generators of $\sigma$ lie on an affine hyperplane $(R^*=1)$ and thus $(R^*=1)\cap \sigma$ is a lattice polytope that we denote by $P$. 
The torus action on $X$ induces a lattice grading on the tangent space $T^1_X$ (resp. $T^1_{\paar}$) and the obstruction space $T^2_X$ (resp. $T^2_{\paar}$) of the deformation functor of $X$ (resp. $\paar$). 
If $X$ has an isolated singularity, then $T^1_X$ is finite dimensional and lies in the single lattice degree $-R^*$. 
In the case where $X$ has non-isolated singularities, the tangent spaces $T^1_X$ and $T^1_{\paar}$ are infinite-dimensional. For the definition and basic properties of the miniversal deformation space of $\paar$ in the three-dimensional case, see e.g.\ \cite[Remark 2.1]{mojcor}. In this paper we focus on special finite-dimensional parts of the tangent space $T^1_{\paar}$ and thus avoid problems that come with infinite dimension. More precisely, let 
$B$ be the set of elements 
$m\in \widetilde{M}$ such that $m$ takes value $1$ on some face $G$ of $P$ and it has value strictly less than $1$ for any lattice point on $P$ lying outside $G$, cf.\ \eqref{eq minb}.
For any $m\in B$ we are going to construct a deformation of $\paar$ which is maximal with prescribed tangent space $\bigoplus_{k\in \NN}T^1_{\paar}(-km)\subset T^1_{\paar}$, i.e.\ we cannot extend it to a deformation of $\paar$ with a larger base space by keeping its tangent space fixed. We say that this deformation is miniversal in degrees $-km$, $k\in \NN$, see Section \ref{sec obs}. For each $k\in \NN$ we thus get a homogeneous deformation in degree $-km$. We can also present those homogeneous deformations in functorial language and the differential graded Lie algebra that controls it is the Harrison differential graded Lie algebra restricted to the degree $-km$, see \cite{fil}.

In the following we are going to emphasize the main differences between this paper and the papers \cite{alt}, \cite{kas}, \cite{m} which also construct some miniversal deformation spaces. Since $X$ is Gorenstein, the results of  \cite{alt}, \cite{kas}, \cite{m} can only be applied to the case when $P$ has edges of lattice length $1$. More precisely, they all consider homogeneous deformations in a primitive lattice degree $-m\in \tM$, with $m\in \sigma^\vee$ and some assumptions on the edges of $\sigma\cap (m=1)$, which imply that the results in the Gorenstein case work only for $m=R^*$ and $P$ having all edges of lattice length $1$, which is equivalent to $X$ being smooth in codimension $2$. The assumption $m\in \sigma^\vee$ in \cite{alt}, \cite{kas}, \cite{m} is very  crucial, since the homogeneous deformation in degree $-m$ is constructed by  proving flatness of some semigroups and this cannot be obtained if $m\not\in \sigma^\vee$. 
However, in the Fano classification problems mentioned above, it is essential to understand deformations of $X$ associated to an arbitrary lattice polytope $P$.

In this paper we drop both of the assumptions made in the papers \cite{alt}, \cite{kas}, \cite{m}, which means we consider also degrees $m\not\in \sigma^\vee$ and $P$ arbitrary. Note that, when considering deformations of the affine cone in order to deform the corresponding Fano Gorenstein toric variety, we require that $m \notin \sigma^\vee$ and that $m$ takes value $0$ on the point $(\underline{0}, 1) \in \widetilde{N}$, where $\underline{0} \in P \subset N_\QQ$ is the unique interior lattice point of $P$. In Remark \ref{proj rem} we describe the cases where our deformations deform also the corresponding Fano Gorenstein toric variety, which is not possible with the results obtained in \cite{alt}, \cite{kas}, \cite{m}. 

The main result of this paper is the following:

\begin{theorem}
Let $m \in B$, and let $X$ be an affine Gorenstein toric variety associated to a lattice polytope $P$. We construct a deformation of $\paar$ which is miniversal in degrees $-km$, for all $k \in \NN$. The irreducible components of the reduced miniversal deformation space are in one-to-one correspondence with maximal Minkowski decompositions of the polytope $P \cap (m = 1)$.
\end{theorem}

For any $m\in B\subset \tM$ we construct a deformation of $\paar$ in Section \ref{sec const def}. The main result of this section  is Theorem \ref{th flat x map}, where we prove that our constructed family is flat 
directly by lifting relations among equations of $X$ and not by bypassing the problem to flatness of semigroups as it was done in \cite{alt}, \cite{kas}, \cite{m}. This in particular enables us to work also with homogeneous deformations in degrees $-m$ with $m\not\in \sigma^\vee$. 
We describe the tangent space $T^1_{\paar}$ and the obstruction space $T^2_{\paar}$ in Subsection \ref{the tang and obs} and we show that the tangent space of our base space equals $\bigoplus_{k\in \NN}T^1_{\paar}(-km)$ in Subsection \ref{sub sec of b}. Since the edges of $P$ might have lattice length greater than $1$, we see that it is more natural to consider homogeneous deformations in degrees $-k m$ for all $k\in \NN$. Note that if $P$ has edges of lattice length $1$, then $T^1_{\paar}(-km)=T^1_X(-km)=0$ for all $k\geq 2$ and thus we only get a homogeneous deformation in degree $-m$.

Moreover, considering the deformations of $\paar$ instead of only $X$ in fact simplifies the construction since the tangent space of the naturally constructed deformation family has dimension 
$$\dim_\CC \bigoplus_{k\in \NN}T^1_{\paar}(-km)=1+\dim_\CC \bigoplus_{k\in \NN}T^1_X(-km).$$ In fact, even in the isolated case (with $m = R^*$), the tangent space of the deformation family constructed from the aforementioned semigroups has dimension $1 + \dim_\CC T^1_X = 1 + \dim_\CC T^1_X(-R^*)$. Thus, non-trivial computations are required to obtain a tangent space of the correct dimension and to prove the bijectivity of the Kodaira–Spencer map. In this paper we see that even if we are only interested in deformations of $X$ it is better to consider deformations of $\paar$ and then apply the obvious  forgetful functor.

Finally, to prove miniversality, we prove that the Kodaira--Spencer map of our constructed deformation family is bijective (which is proven in Section \ref{kod spe sec} in a completely different way than in the previously mentioned papers) and that the obstruction map is surjective (which is proven in Section \ref{sec obs}).  We conclude the paper by showing that the reduced irreducible components of the constructed deformation of $(X,\partial X)$ 
are in one to one correspondence with maximal Minkowski decompositions of the polytope $P\cap (m=1)$ in Section \ref{sec irr compo}.

\section*{Acknowledgement}
I am very grateful to Alessio Corti for many stimulating conversations and for  suggesting to work with deformations of pairs. I also wish to thank to  
Klaus Altmann, Alexandru Constantinescu, Nathan Ilten and Andrea Petracci for many useful conversations.

\section{Preliminaries}\label{sec free pairs}

\subsection{The setup}\label{sub the setup}
We fix $\CC$ to be an algebraically closed field of characteristic $0$. Let $P$ be a lattice polytope with vertices $v^1,\dots,v^p$ in $N$, where $N$ is a lattice. Putting $P$ on height $1$ gives us a  rational, polyhedral cone $$\sigma=\lan a^1,\dots,a^p \ran\subset \tN_\QQ=(N\oplus \ZZ)\otimes_\ZZ \QQ$$ with $a^i=(v^i;1)\in\tN$, $i=1,\dots,p$.  Let $M$ denote the dual lattice of $N$ and let us consider the monoid $S=\sigma^{\vee}\cap \tM=\sigma^{\vee}\cap (M\oplus \ZZ)$, where $\sigma^{\vee}:=\{r\in \tM_\QQ~|~\lan \sigma,r \ran\geq 0\}$. Every affine Gorenstein toric variety is isomorphic to $X:=X_P:=\spec \CC[S]$ for some lattice polytope $P$, where $\CC[S]:=\bigoplus_{u\in S}\chi^u$ is the semigroup algebra.

Let $E_1,E_2,\dots,E_n$ be the edges of $P$. We denote the lattice length of an edge $E_i$ by $\ell(E_i)\in \NN$, for $i=1,\dots,n$. We equip every edge $E_i=[w^i,z^i]$, connecting two vertices $w^i$ and $z^i$, with an orientation and present it as a vector 
\begin{equation}\label{eq pop1}
d^i:=w^i-z^i\in N.
\end{equation}

\begin{definition}\label{d:eta(c)}
For any face $G\subset P$ (including $G=P$) and
for $c\in M$ we choose a vertex $v_G(c)$ of $G$ where $\lan c,\cdot \ran$ becomes minimal on $G$. 
We define $\eta_G(c) :=-\min_{v\in G}\langle v,c\rangle=
-\langle v_G(c),c\rangle \in \ZZ$, which is independent of the choice of $v_G(c)$. 
If $G=P$ we also write $\eta(c)$ for $\eta_P(c)$ and $v(c)$ for $v_P(c)$.
\end{definition}

The Hilbert basis of $S=\sigma^{\vee}\cap \tM$ is equal to
\begin{equation}\label{eg hilbbas}
E:=\big\{s_1=(c_1;\eta(c_1)),\dots,s_r=(c_r;\eta(c_r)),R^*=(\un{0};1)\big\},
\end{equation}
with uniquely determined elements $c_i\in M$ (see e.g.\ \cite[Section 4.3]{alt}).

\subsection{Equations of $S$ and their linear relations}\label{subsec 1.2}
  By \eqref{eg hilbbas}, we obtain $X=\spec \CC[u,x_1,\dots,x_r]/\cI_S$, where $\cI_S$ is the kernel of the map $\CC[u,x_1,\dots,x_r]\to \CC[S]$, $u\mapsto R^*$, $x_j\mapsto s_j$.

Note that  we can write every element $s\in S=\sigma^\vee\cap \tM$ in a unique way as $s=\partial(s)+nR^*$, where $n\in \NN$ and $\partial(s)\in \partial(S)$ is an element of the boundary of $S$, defined as
\begin{equation}\label{eq partial s}
\partial(S):=\{s\in S~|~s-R^*\not\in S\}.
\end{equation}

For every $\bfk=(k_1,\dots,k_r)\in \NN^r$ we define
\begin{equation}\label{eq etagbfk}
\eta_G(\bfk) := \eta_G(k_1c_1)+\eta_G(k_2c_2)+\cdots +\eta_G(k_rc_r) - \eta_G(k_1c_1+k_2c_2+\cdots +k_rc_r)
\in\NN
\end{equation}
and we write $s_\bfk:=\sum_{i=1}^rk_is_i\in S\subset \tM$ and $c_{\bfk}:=\sum_{i=1}^rk_ic_i\in M$. We immediately see that there is a unique decomposition 
$s_{\bfk}=\bo(\bfk)+\eta_P(\bfk)R^*$
with $\bo(\bfk)=(c_\bfk;\eta_P(c_\bfk))\in \partial(\kS)$.

For every $\bfk\in \NN^r$ we choose $b_i\in \NN$ such that 
$
 \bo(\bfk)=\sum_{i=1}^rb_is_i 
$
 and denote
$$\bfx^{\bfk}:=\prod_{i=1}^rx^{k_i}_i,~~~\bfx^{\bo(\bfk)}:=\prod_{i=1}^rx_i^{b_i}.$$
We define the binomials 
\begin{equation}\label{eq fbfk xu2}
f_\bfk(u,\bfx):=\bfx^\bfk - \bfx^{\bo(\bfk)}\,u^{\eta_P(\bfk)}\in \CC[u,\bfx]:=\CC[u,x_1,\dots,x_r]
\end{equation}
and the following lemma shows that they generate the ideal $\cI_S\subset \CC[u,x_1,\dots,x_r]$.

\begin{lemma}\label{lem 29}
It holds that $\cI_{S}=(f_{\bfk}~|~\bfk\in \NN^r)$ and 
the module of linear relations between the $f_\bfk$, which is the kernel of the map
\[
  \textstyle
  \psi:\bigoplus_{\bfk\in \NN^r}\CC[u,x_1,\dots,x_r]e_\bfk\xrightarrow{e_\bfk\mapsto f_\bfk} \cI_\kS\subset \CC[u,x_1,\dots,x_r],
\]
 is spanned by $R_{\bfa,\bfk}:=e_{\bfa+\bfk} - \bfx^\bfa e_\bfk - u^{\eta_P(\bfk)}e_{\bo(\bfk)+\bfa}$, for $\bfa,\bfk\in \NN^r$. 
\end{lemma}
\begin{proof}
The first statement follows immediately by definition, the proof of the second one is the same as the proof of \cite[Lemma 5.6]{m}.
\end{proof}

\begin{example}\label{ex house pop}
 Consider the polytope $P=\conv\{(0,0),(2,0),(2,2),(1,3),(0,2)\}$, where $\conv$ denotes the convex hull.
 $$
 \begin{tikzpicture}[scale=0.8]
\draw[->,thick,  color=black] 
(0,0) -- (2,0);
\draw[->,thick,  color=black] 
(2,0) -- (2,2);
\draw[->,thick,  color=black] 
(2,2) -- (1,3);
\draw[->,thick,  color=black] 
(1,3) -- (0,2);
\draw[->,thick,  color=black] 
(0,2) -- (0,0);
\fill[thick,  color=black]
  (0,0) circle (2.5pt) (2,0) circle (2.5pt) (2,2) circle (2.5pt) (1,3) circle (2.5pt)
  (0,2) circle (2.5pt) (0,1) circle (2.5pt) (1,0) circle (2.5pt) (2,1) circle (2.5pt);
\draw[thick,  color=black]
(-.5,1) node{$d^3$} (0.2,2.8) node{$d^2$} (1.8,2.8) node{$d^1$} (2.5,1) node{$d^5$} (1,.4) node{$d^4$} (-.3,-.3) node{$v^4$} (2.3,-.3) node{$v^5$} (2.3,2.2) node{$v^1$} (1,3.3) node{$v^2$} (-.3,2.2) node{$v^3$};
\end{tikzpicture}
$$
The following notation will be used in the examples that follow.
We denote oriented edges of $P$ by
$$d^1=(-1,1),~d^2=(-1,-1),~d^3=(0,-2),~d^4=(2,0),~d^5=(0,2)$$
and vertices of $P$ by 
$$v^1=(2,2),~~v^2=(1,3),~~v^3=(0,2),~~v^4=(0,0),~~v^5=(2,0).$$

The Hilbert basis $E$ of $S$ is in this case equal to
$$E=\{(0,1;0),(-1,0;2),(-1,-1;4),(0,-1;3),(1,-1;2),(1,0;0),(0,0;1)\},$$
i.e. $c_1=(0,1)$, $\eta(c_1)=0$, $c_2=(-1,0)$, $\eta(c_2)=2$,\dots, $c_6=(1,0)$, $\eta(c_6)=0$.

Let $\bfk_1:=e_4+e_6:=(0,0,0,1,0,1)\in \NN^6$, $\bfk_2=e_3+e_6$ and 
$$\bfk_3=e_3+e_5,~~~\bfk_4=e_2+e_5,~~~\bfk_5=e_2+e_4,~~~\bfk_6=e_2+e_6,~~~\bfk_7=e_1+e_3,~~~\bfk_8=e_1+e_5,~~~\bfk_9=e_1+e_4.$$
This gives us 
$$f_{\bfk_1}=x_4x_6-x_5u,~~~f_{\bfk_2}=x_3x_6-x_4u,~~~f_{\bfk_3}=x_3x_5-x_4^2,~~~f_{\bfk_4}=x_2x_5-x_4u,$$
$$f_{\bfk_5}=x_2x_4-x_3u,~~~f_{\bfk_6}=x_2x_6-u^2,~~~f_{\bfk_7}=x_1x_3-x_2u^2,~~~f_{\bfk_8}=x_1x_5-x_6u^2,~~~f_{\bfk_9}=x_1x_4-u^3.$$
\end{example}

\begin{remark}
By direct computer calculation, we can verify that the polynomials $f_{\bfk_i}$, for $i = 1, \dots, 9$, are  the generators of $\cI_S$; for example, using the following \texttt{Macaulay2} code:
\begin{verbatim}
A = matrix{{0,-1,-1,0,1,1,0},{1,0,-1,-1,-1,0,0},{0,2,4,3,2,0,1}}
M = toricMarkov(A)
R = QQ[x_1,x_2,x_3,x_4,x_5,x_6,u]
I = toBinomial(M,R)
\end{verbatim}
\end{remark}

\subsection{The monoid $\ktT(G)$}\label{subsec geo}

Recall that $E_1, \dots, E_n$ are the edges of $P$, and without loss of generality, we assume that $E_1, \dots, E_{n_G}$ are the edges of a face $G \subset P$, for some $n_G \in \NN$.

Recall the vectors $d^i$ from \eqref{eq pop1}. In what follows, we construct an important vector space $\cT(G)$, which was already defined in \cite[Section~2.2]{alt}.
We choose an orientation for every $2$-face $\eps$ of $G$: let $\delta_\eps(d^i)\in\{0,1,-1\}$ with the property that $\delta_\eps(d^i)=0$ if $d^i\not\in\eps$ and $\delta_\eps(d^i)\in \{-1,1\}$ if $d^i\in \eps$ and moreover we require $\sum_{d^i\in \eps}\delta_\eps(d^i)\cdot d^{i}=0$. Since $\delta_\eps(d^i)$ is defined with respect to the face $\eps$, such an orientation always exists, as the boundary of each $2$-face is a closed polygonal cycle.  
We define the vector space
$$\cT(G)=\{(t_{1},\dots,t_{n_G})\in \QQ^{n_G}~|~\sum_{d^i\in \eps}\delta_\eps(d^i)t_{i}d^{i}=0\text{ for every 2-face }\eps\text{ in } G\}.$$
\begin{definition}
We define the lattice $\kMTZ(G)\subset\kMT(G)$ by
$$
  (t_1,\dots,t_{n_G})\in \kMTZ(G):\iff
    t_{i}d^i \in N \text{  for each }i=1,\dots,n_G.
$$
Moreover, let us define the monoid
$$\tT(G):=\span_\NN\{\ell(d^1)t_1,\dots,\ell(d^n)t_{n_G}\}\subset \cT^*_\ZZ(G),$$ where $\ell(d^i) = l_i$ denotes the lattice length of the oriented edge $d^i$, and  $\cT^*_\ZZ(G)$ is the dual lattice of $\cT_\ZZ(G)$. Note that the $t_i$ serve as coordinate functions on $\cT(G)$, and thus correspond to elements of the dual space $\cT^*(G)$.

\end{definition}

Without loss of generality assume that one of the vertices of $G$ is equal to $0\in N$ (note that this may require shifting $P$ by a lattice vector). For $c \in M$, recall the vertex $v_G(c)$ of $G$ as defined in Definition \ref{d:eta(c)}.
\begin{definition}
\label{def-etsTildeZ}
Let $c \in M$ and let us choose a path along the edges of $G$, going through vertices  
$
w^1 = 0, \; w^2, \; \ldots, \; w^k = v_G(c),
$
where  $k-1$ is the number of edges in the path, i.e. $[w^{j}, w^{j+1}]$ is an edge of $G$ for every $j = 1, \dots, k-1$.
We denote this path by $p_{w^1 \leadsto w^k}$.

For every $c\in M$ we define 
$$ 
\tilde{\eta}_G(c):=\sum_{j=1}^{k-1} \lan w^{j}-w^{j+1},c\ran \cdot t_{[w^j,w^{j+1}]}\in \cT^*_\ZZ(G),
$$
where $\cT^*_\ZZ(G)$ is the dual lattice of $\cT_\ZZ(G)$ and $t_{[w^j,w^{j+1}]}$ correspond to the edge $[w^j,w^{j+1}]$, i.e.\ if $E_i=[w^j,w^{j+1}]$, then $t_i=t_{[w^j,w^{j+1}]}$.

Similarly as in \eqref{eq etagbfk}, for $\bfk=(k_1,\dots,k_r)\in \NN^r$ we also define
$$
\te_G(\bfk) := \te_G(k_1c_1)+\te_G(k_2c_2)+\cdots +\te_G(k_rc_r) - \te_G(k_1c_1+k_2c_2+\cdots +k_rc_r)
\in\cT_\ZZ^*(G).
$$
\end{definition}

\begin{lemma}\label{lem gen kt gor}
It holds that $\te_G(\bfk)\in \ktT(G)$. 
\end{lemma}
\begin{proof}
Let $c:=k_1c_1+\cdots+ k_rc_r$. For each $j\in \{1,\dots,r\}$ we pick a path $p_{v_G(c)\leadsto v_G(c_j)}$ along the edges of $G$ starting in $v_G(c)$ and ending in $v_G(c_j)$:
 $$v^1_j=v_G(c),~v^2_j,\dots,v^{p_j}_j=v_G(c_j),$$ such that $\lan v_j^l-v_j^{l+1},c_j\ran\geq 0$ for all $l\in \{1,\dots,p_j-1\}$. Note that it is always possible to pick such a path since $c_j$ achieves minimum on $G$ at $v(c_j)$. 
 Thus for computing $\te_G(c)$ we choose an arbitrary path from $0$ to $v(c)$ and for computing $\te_G(k_jc_j)$ we pick first the previous path from $0$ to $v(c)$  and then the above path $p_{v_G(c)\leadsto v_G(c_j)}$, for all $j=1,\dots,r$.
Note that, by the definition of $\cT^*_\ZZ(G)$, every choice of the above paths yields the same element $\te_G(\bfk)\in \cT^*_\ZZ(G)$.
 Thus we have 
\begin{equation}\label{eq main eq}
\te_G(\bfk)=\sum_{j=1}^r\Big(\sum_{l=1}^{p_j-1}\lan v^l_j-v^{l+1}_j,k_jc_j\ran t_{[v^l_j,v^{l+1}_j]}\Big).
\end{equation}
The coefficients before $t_{[v^l_j,v^{l+1}_j]}$ are either zero or positive multiples of $\ell([v^l_j,v^{l+1}_j])$, which proves the claim.
\end{proof}

\begin{definition}\label{def map pi}
We define the \emph{degree map} $\deg: \cT^*_\ZZ(G)\otimes_\ZZ\QQ\to \QQ$, which maps all $t_i$ to $1$. 
\end{definition}

For $c\in M$ we see that 
\begin{equation}\label{eq deg te}
\deg(\te_G(c))=\sum_{j=1}^{k-1}\lan v^{j}-v^{j+1},c\ran=-\lan v_G(c),c\ran=\eta_G(c)\in \NN,
\end{equation}
where we used that $v^1=0$ and $v^k=v_G(c)$ as in Definition \ref{def-etsTildeZ}. 
\begin{corollary}\label{cor zapiszi}
We can choose $z_i(\bfk)\in \NN$ such that  $\te_G(\bfk)=\sum_{i=1}^{n_G}l_iz_i(\bfk)t_i\in \ktT(G)$, where recall that $l_i=\ell(d^i)$ is the lattice length of the edge $d^i$. For any such choice of $z_i(\bfk)$ it holds that 
\begin{equation}\label{eq uiuli}
\eta_G(\bfk)=\sum_{i=1}^{n_G}l_iz_i(\bfk).
\end{equation}
\end{corollary}
\begin{proof}
It follows immediately by Lemma \ref{lem gen kt gor} and \eqref{eq deg te}.
\end{proof}

 Note that the choice of $z_i(\bfk)$ from \eqref{eq uiuli} is not uniquely determined, since we can for example choose different paths $p_{v_G(c)\leadsto v_G(c_j)}$ in the proof of Lemma \ref{lem gen kt gor}, from which we obtained $z_i(\bfk)$.

\begin{example}\label{ex 2.7}
In Example \ref{ex house pop} (where $G=P$ and we write $\te$ for $\te_P$) we 
first choose $$v(c_1)=v^4=(0,0),~~~ v(c_2)=v^5=(2,0),~~~v(c_3)=v^1=(2,2),$$
$$v(c_4)=v^2=(1,3),~~~v(c_5)=v^3=(0,2),~~~v(c_6)=v^4=(0,0)$$
and thus
$$\te(c_1)=0,~~~\te(c_2)=2t_4,~~~\te(c_3)=2t_4+2t_5,~~~\te(c_4)=2t_3+t_2,~~~\te(c_5)=2t_3,~~~\te(c_6)=0.$$
Moreover, for $\bfk_1=e_4+e_6$ we have
$$\te(\bfk_1)=\te(c_4)+\te(c_6)-\te(c_4+c_6)=\te(c_4)+\te(c_6)-\te(c_5)=t_2.$$
For $\bfk_2=e_3+e_6$ we have 
$$\te(\bfk_2)=\te(c_3)+\te(c_6)-\te(c_4)=2t_4+2t_5-2t_3-t_2=t_2,$$
since $\te(c_3)=2t_4+2t_5=2t_3+2t_2\in \ktT(P)$. Note that we also obtain $\te(\bfk_2)=t_2$
by \eqref{eq main eq} since
 the path $p_{v(c_4)\leadsto v(c_3)}$ is going through $v^1_3=v(c_4)=v^2,~v^2_{3}=v(c_3)=v^1$ and
$\lan v^2-v^1,c_3\ran=0$ and moreover, the path
$p_{v(c_4)\leadsto v(c_6)}$ is going through $v^1_6=v(c_4)=v^2$, $v^2_6=v^3,$ $v^3_6=v(c_6)=v^4$ and $\lan v^2-v^3,c_6\ran=1$ and $\lan v^3-v^4,c_6\ran=0$. 
In the same way we compute
 $$\te(\bfk_3)=0,~~~\te(\bfk_4)=t_1,~~~\te(\bfk_5)=t_1,~~~\te(\bfk_6)=2t_4,~~~\te(\bfk_7)=2t_5,~~~\te(\bfk_8)=2t_3,~~~\te(\bfk_9)=t_2+2t_3,$$
from which we see that $z_2(\bfk_9)=1$, $z_3(\bfk_9)=1$, $z_i(\bfk_9)=0$ for $i=1,4,5$ and $z_3(\bfk_8)=1$, $z_i(\bfk_8)=0$ for $i=1,2,4,5$. 
Note that, with a different choice of paths, we may have $\te(c_4) = 2t_5 + t_1$ and thus $\te(\bfk_9) = t_1 + 2t_5$, which leads to a different choice of $z_i(\bfk_9)$ as described above, since in this case we have $z_1(\bfk_9) = z_5(\bfk_9) = 1$ and $z_i(\bfk_9) = 0$ for all $i = 2, 3, 4$.
\end{example}

\section{Flatness}\label{sec const def}
In this section we construct a flat family that is deforming $\paar$.
For the following definition see also \cite[Section 3.4.4]{ser}.
\begin{definition}\label{def mini}
A \emph{deformation} of $X$ is a flat family of schemes $f:\cX\to \cS$ with $0\in \cS$ such that $f^{-1}(0)=X$.
A \emph{deformation of a pair} $\paar$ is a \emph{deformation of a closed embedding} $\partial X\hookrightarrow X$, which is a diagram:
\begin{equation}\label{eq def diagram prvi}
\begindc{\commdiag}[500]
\obj(0,1)[01]{$\widetilde{\cY}$}
\obj(1,0)[10]{$\widetilde{\cS}$}
\obj(2,1)[21]{$\widetilde{\cX}$}
\mor{01}{10}{$f_1$}[\atright,\solidarrow]
\mor{01}{21}{}[\atleft, \injectionarrow]
\mor{21}{10}{$f_2$}
\enddc
\end{equation}
with $f_i$ flat for $i\in \{1,2\}$ and $f_1^{-1}(0)=\partial X$, $f^{-1}_2(0)=X$. 
\end{definition}

It is straightforward to define a deformation functor $F_X$ (resp. $F_{\paar}$) to be the isomorphism classes of deformations of $X$ (resp. $\paar$) over $\spec A$, where $A$ is artinian local $\CC$-algebra with residue field $\CC$ (see e.g.\ \cite[Section 3.4.4]{ser}).
The corresponding tangent spaces we denote by $T^1_X$ and $T^1_{\paar}$. We are going to analyse those tangent spaces as well as the obstruction spaces in Subsection \ref{the tang and obs}.

\subsection{Homogeneous deformations}\label{sub hom def}

Let $m\in \tM$ be such that $(m=1)\cap (R^*=1)\cap \sigma$ equals a face $G$ of $P$, which is not a vertex (here we put $P$ on height $1$, i.e.\ $P=(R^*=1)\cap \sigma$), and moreover, $m$ has value $< 1$ for any lattice point on $P$ lying outside $G$. If $m=R^*$, then $G=P$.  Note that 
\begin{equation}\label{eq minb}
m\in B:=\{R^*-s~|~s\in \partial(S),~(s=0)\cap P~\text{is a face of $P$, which is not a vertex}\}
\end{equation}
and thus every such $m$ can be written as 
 $m=R^*-\sum_{i=1}^rn_is_i\in \tM$ for some $n_i\in \NN$. 
Assume that one of the vertices of $G\subset P\subset N$ is equal to $0\in N$.
We are going to construct a deformation of an affine toric variety $X_P$ using the monoid  $\ktT(G)$ (with the deformation parameters having degrees $km\in \tM$ for $k\in \NN$).

We embed $\spec \CC[\ktT(G)]$ into $\spec \CC[u_1,\dots,u_{n_G}]=\spec \CC[\bfu]$, where $n_G$ denotes the number of edges on $G\subset P$. We denote the kernel of the map $\CC[u_1,\dots,u_{n_G}]\to \CC[\ktT(G)],~u_i\mapsto \tilde{t}_i$ by 
\begin{equation}\label{eq cItitt}
\cI_{\ktT(G)}\subset \CC[u_1,\dots,u_n].
\end{equation}
We choose $z_i(\bfk)\in \NN$ such that $\te_G(\bfk)=\sum_{i=1}^{n_G}l_iz_i(\bfk)t_i\in \ktT(G)$, cf.\ Corollary \ref{cor zapiszi}, and denote
\begin{equation}\label{eq gbfk}
\bfu^{\te_G(\bfk)}:=u^{\eta_P(\bfk)-\eta_G(\bfk)}\prod_{i=1}^{n_G}u^{z_i(\bfk)}_i\in \CC[u,u_1,\dots,u_{n_G}]=\CC[u,\bfu].
\end{equation}
Moreover, let us fix a representation $m=R^*-\sum_{i=1}^rn_is_i\in \tM$ with $n_i\in \NN$ and define 
$$\bfx_m:=\prod_{i=1}^{r}x_i^{n_i}.$$
Note that if $m=R^*$, then $\bfx_m=1$.

We also introduce variables $T_{ij}$, $i=1,\dots,n_G$, $j=1,\dots,l_i=\ell(E_i)$  of degrees $\deg T_{ij}=jm\in \tM$ for all $i$ and define the maps 
\begin{equation}\label{eq ui new 2}
f_{\bfu\to \bfT}:\CC[\bfu]\to \CC[\bfx_m,u,\bfT],~~~
f_{\bfu\to \bfT}(u_{i}):=u^{l_i}+\sum_{j=1}^{l_i}\Big(\bfx^j_mu^{l_i-j}T_{ij}\Big).
\end{equation}
\begin{equation}\label{eq ui new 34}
g_{\bfu\to \bfT}:\CC[u,\bfu,\bfx]\to \CC[u,\bfT,\bfx],~~~u_{i}\mapsto f_{\bfu\to \bfT}(u_{i}),~~x_j\mapsto x_j,~~u\mapsto u.
\end{equation}

\begin{remark}
The variables of our base space will be $T_{ij}$ and that this choice of $u_i$ is the most natural one will become clear in Section \ref{kod spe sec} (more precisely in the proof of Proposition \ref{prop 32}) where we will see that the base space has the right dimension with this choice. From the same proof it will also become clear why we are  working with $m\in B$. Note also that $u^{l_i}$ and $\bfx^j_mu^{l_i-j}T_{ij}$ (for all $j$) have degree $l_iR^*$  and thus $f_{\bfu\to \bfT}(u_{i})$ is homogeneous.
\end{remark}

Let us denote $$\CC[\bfT]:=\CC[T_{ij}~|~i\in\{1,\dots,n_G\},~(i,j)\in \{(i,1),\dots,(i,l_i)\}]=\CC[T_{11},\dots,T_{1l_1},\dots,T_{n_G1},\dots,T_{n_Gl_{n_G}}]$$ and define
\begin{equation}\label{eq fbfk}
F_{\bfk}(\bfu,\bfx):=\bfx^{\bfk}-\bfx^{\bo(\bfk)}\bfu^{\te_G(\bfk)}\in \CC[u,\bfu,\bfx].
\end{equation}
Let $\cI_F:=(F_{\bfk}~|~\bfk\in \NN^r)\subset \CC[u,\bfu,\bfx]$ be the ideal generated by $F_{\bfk}$. 
We denote by $\cJ_{\ktT(G)} \subset \CC[u, \bfT, \bfx_m]$ (resp.\ $\cJ_F \subset \CC[u, \bfT, \bfx]$) the ideal generated by $f_{\bfu \to \bfT}(\cI_{\ktT(G)})$ (resp.\ $g_{\bfu \to \bfT}(\cI_F)$).

\begin{remark}\label{rem depg}
Note that  the ideals $\cJ_{\ktT(G)}$ and $\cJ_F$ depend on $m$ but we keep the notation simple and do not write additional subscript $m$.
\end{remark}

\begin{example}\label{ex zad exex}
From Example \ref{ex 2.7} we see that
 $$F_{\bfk_1}=x_4x_6-x_5u_2,~~~F_{\bfk_2}=x_3x_6-x_4u_2,~~~F_{\bfk_3}=x_3x_5-x_4^2,~~~F_{\bfk_4}=x_2x_5-x_4u_1,$$
$$F_{\bfk_5}=x_2x_4-x_3u_1,~~~F_{\bfk_6}=x_2x_6-u_4,~~~F_{\bfk_7}=x_1x_3-x_2u_5,~~~F_{\bfk_8}=x_1x_5-x_6u_3,~~~F_{\bfk_9}=x_1x_4-u_2u_3$$
and we obtain $f_{\bfk_i}$ from $F_{\bfk_i}$ if we write $u_i=u^{l_i}$.
\demo
\end{example}
For an integer $z\in \ZZ$ we define 
$$z^+:=\left\{
\begin{array}{cc}
z&\text{ if }z\geq 0\\
0&\text{otherwise},
\end{array}
\right.
z^-:=\left\{
\begin{array}{cc}
-z&\text{ if }z\leq 0\\
0&\text{otherwise}.
\end{array}
\right.
$$
It is clear that 
\begin{equation}\label{eq pc}
\cI_{\ktT(G)}=\Big(\prod_{i=1}^{n_G} u_i^\frac{d_i^+}{l_i}-\prod_{i=1}^{n_G}u^{\frac{d_i^-}{l_i}}_i~|~\un{d}\in \cT^*_\ZZ(G) \cap \cT(G)^\perp\Big)\subset \CC[\bfu]=\CC[u_1,\dots,u_{n_G}]
\end{equation}
with 
\begin{equation}\label{eq ctp perp}
\cT(G)^\perp=\span_\QQ\Big\{\Big(\delta_\eps(d^1)\lan d^1,c\ran,\dots,\delta_\eps(d^n)\lan d^n,c\ran\Big)~|~c\in M_\QQ, \eps~\text{a 2-face in }G\Big\}.
\end{equation}

Let $\un{d}\in \cT^*_\ZZ(G) \cap \cT(G)^\perp$ (as in the equation \eqref{eq pc}) and let
 \begin{equation}\label{pund}
 p_{\un{d}}(\bfu):=\prod_{i=1}^{n_G} u_i^\frac{d_i^+}{l_i}-\prod_{i=1}^nu^{\frac{d_i^-}{l_i}}_i\in \cI_{\ktT(G)}\subset \CC[\bfu].
 \end{equation}
 Since $p_{\un{d}}(\bfu)$ is homogeneous of degree $g_{\un{d}}R^*$, where 
 $g_{\un{d}}=\sum_{i=1}^{n_G}d^+_i=\sum_{i=1}^{n_G}d^-_i$, we can  write in a unique way
\begin{equation}\label{eq pol p}
f_{\bfu\to \bfT}(p_{\un{d}}(\bfu))=\sum_{j=1}^{g_{\un{d}}}\bfx_m^ju^{g_{\un{d}}-j}p^{(j)}_{\un{d}}(\bfT),
\end{equation}
where $p^{(j)}_{\un{d}}(\bfT)\in \CC[\bfT]$ are homogeneous of degree $jm\in \tM$. 
We define the ideal 
$$\cJ_\cB:=\lan p^{(j)}_{\un{d}}(\bfT)\mid \un{d}\in \cT^*_\ZZ(G) \cap \cT(G)^\perp,~~~j=1,\dots,g_{\un{d}}\ran\subset \CC[\bfT],$$ 
i.e.\ $\cJ_\cB$ is generated by the polynomials $p^{(j)}_{\un{d}}(\bfT)$ for all $\un{d}\in \cT^*_\ZZ(G) \cap \cT(G)^\perp$ and $j=1,\dots,g_{\un{d}}$. 

\begin{remark}
Note that the term $u^{g_{\un{d}}}$ gets cancelled in \eqref{eq pol p}.
\end{remark}

\begin{example}\label{ex house}
Let us continue our Example \ref{ex 2.7}. Using the notation from Subsection \ref{sub hom def} we take $m=R^*$ and thus we have $G=P$.
 We see that 
the ideal $\cI_{\ktT(P)}$ is in this case generated by 
$$\cI_{\ktT(P)}=\lan u_4-u_1u_2,~~~u_5u_1-u_2u_3\ran.$$
 The two generators are obtained from $$ \big(\lan d^1,(1,0)\ran,\dots,\lan d^5,(1,0)\ran\big),~~~~\big(\lan d^1,(0,1)\ran,\dots,\lan d^5,(0,1)\ran\big)\in \cT^*_\ZZ(P) \cap \cT(P)^\perp.$$ 
 Thus $\CC[\bfu]/\cI_{\ktT(P)}\cong \CC[u_1,u_2,u_3,u_5]/(u_5u_1-u_2u_3)$ and 
$f_{\bfu\to \bfT}(u_5u_1-u_2u_3)$ is
$$\big(u^2+T_{52}+uT_{51}\big)(u+T_{11})-(u+T_{21})\big(u^2+T_{32}+uT_{31}\big)=$$
$$=u^2\big(T_{11}+T_{51}-T_{21}-T_{31} \big)+u\big(T_{11}T_{51}+T_{52}-T_{32}-T_{21}T_{31} \big)+T_{11}T_{52}-T_{21}T_{32}.$$
Thus the ideal $\cJ_\cB$ of $\cB\subset \spec \CC[T_{11},T_{21},T_{31},T_{32},T_{51},T_{52}]$ is given by 
\begin{equation}\label{eq houe ex equ 2}
\cJ_\cB=(T_{11}+T_{51}-T_{21}-T_{31},T_{11}T_{51}+T_{52}-T_{32}-T_{21}T_{31},T_{11}T_{52}-T_{21}T_{32}).
\end{equation}
\end{example}

\subsection{The proof of flatness}
Recall the ideals $\cJ_\cB$ and $\cJ_F$.
\begin{theorem}\label{th flat x map}
The map 
$$\pi_2:\spec \CC[u,\bfT,\bfx]/(\cJ_\cB+\cJ_F)\to \spec \CC[\bfT]/\cJ_\cB$$ is flat.
\end{theorem}
\begin{proof}
Using \eqref{eq gbfk} and \eqref{eq fbfk} we see that 
\begin{equation}\label{eq fbfk flat 2}
F_{\bfk}(u,\bfT,\bfx):=g_{\bfu\to \bfT}\left(F_\bfk(\bfu,\bfx)\right)=\bfx^\bfk-\bfx^{\bo(\bfk)}u^{\eta_P(\bfk)-\eta_G(\bfk)}\prod_{i=1}^{n_G}\Big(u^{l_i}+\sum_{j=1}^{l_i}\Big(\bfx^j_mu^{l_i-j}T_{ij}\Big)\Big)^{z_i(\bfk)}.
\end{equation}
By \eqref{eq uiuli} we see that $F_\bfk(u,\bfT,\bfx)$ is a lift of $f_{\bfk}(u,\bfx)$, which means $F_{\bfk}(u,0,\bfx)=f_\bfk(u,\bfx)$.  
We are going to prove flatness by the lifting relations $R_{\bfa,\bfk}=f_{\bfa+\bfk} - \bfx^\bfa f_\bfk - u^{\eta_P(\bfk)}f_{\bo(\bfk)+\bfa}$ from Lemma \ref{lem 29}. Let 
$\widetilde{R}_{\bfa,\bfk}:=g_{\bfu\to \bfT}\left(F_{\bfa+\bfk}(\bfu,\bfx) - \bfx^\bfa F_\bfk(\bfu,\bfx) - \bfu^{\te_G(\bfk)}F_{\bo(\bfk)+\bfa}(\bfu,\bfx)\right)$ and as before we see that $\widetilde{R}_{\bfa,\bfk}$ is a lift of $R_{\bfa,\bfk}$.

We will show that $\widetilde{R}_{\bfa,\bfk}$ is a linear relation between $F_{\bfk}(u,\bfT,\bfx)$. We compute
\begin{equation}\label{flat eq}
\widetilde{R}_{\bfa,\bfk}=g_{\bfu\to \bfT}\left(-\bfx^{\bo(\bfa+\bfk)}\bfu^{\te_G(\bfa+\bfk)}+ \bfu^{\te_G(\bfk)}\bfx^{\bo(\bfa+\bo(\bfk))}\bfu^{\te_G(\bo(\bfk)+\bfa)}\right).
\end{equation}
Immediately by definition we see that $\bo(\bfa+\bo(\bfk))=\bo(\bfa+\bfk)$ and thus $\bfx^{\bo(\bfa+\bo(\bfk))}=\bfx^{\bo(\bfa+\bfk)}$. In the following we are going to prove that 
\begin{equation}\label{eq ktt rel}
\te_G(\bfa+\bo(\bfk))+\te_G(\bfk)=\te_G(\bfa+\bfk)\in \ktT(G).
\end{equation}
Let us write $\partial(\bfk)=(b_1,\dots,b_r)\in \NN^r$ and thus using $\bo(\bfa+\bo(\bfk))=\bo(\bfa+\bfk)$ we get 
$$\te_G(\bfa+\bfk)-(\te_G(\bfa+\bo(\bfk))+\te_G(\bfk))=\te_G\left(\sum_{i=1}^rk_ic_i\right)-\sum_{i=1}^r\te_G(b_ic_i)=0\in \ktT(G),$$
where the latter equality holds because 
 by definition $b_i=0$ if $v(\sum_{j=1}^rk_jc_j)\ne v(c_i)$ and moreover we have $\sum_{j=1}^rk_jc_j=\sum_{j=1}^rb_jc_j$.
Thus \eqref{eq ktt rel} holds and by applying the degree map, cf.\ Definition \ref{def map pi}, to \eqref{eq ktt rel} we also see that $$\eta_G(\bfa+\bo(\bfk))+\eta_G(\bfk)=\eta_G(\bfa+\bfk)\in \NN,$$ 
for any $G\subset P$ (including $G=P$).
This implies that
$$\bfu^{\te_G(\bo(\bfk)+\bfa)+\te_G(\bfk)}-\bfu^{\te_G(\bfa+\bfk)}\in \cI_{\ktT(G)}\subset \CC[\bfu].$$
As in \eqref{eq pol p} we thus see that \eqref{flat eq} can be in a unique way written as 
\begin{equation}\label{eq obs flat}
\widetilde{R}_{\bfa,\bfk}= \bfx^{\bo(\bfa+\bfk)}u^{\eta_P(\bfa+\bfk)-\eta_G(\bfa+\bfk)}\left(\sum_{j=1}^{\eta_G(\bfa+\bfk)}\bfx_m^ju^{{\eta_G(\bfa+\bfk)}-j}p_{\bfa,\bfk}^{(j)}(\bfT)\right),
 \end{equation}
where $p^{(j)}_{\bfa,\bfk}\in \cJ_\cB\subset \CC[\bfT]$  are homogeneous of degree $jm\in \tM$. 
Thus $\widetilde{R}_{\bfa,\bfk}$ is indeed a linear relation, which finishes the proof by the well known flatness criterion, see e.g.\ \cite[Section 1]{ste}.
\end{proof}

Let us consider the diagram
\begin{equation}\label{eq def diagram}
\begindc{\commdiag}[500]
\obj(0,1)[01]{$\spec \CC[u,\bfT,\bfx]/(\cJ_\cB+\cJ_F,u)$}
\obj(3,0)[10]{$\spec \CC[\bfT]/\cJ_\cB$}
\obj(6,1)[21]{$\spec \CC[u,\bfT,\bfx]/(\cJ_\cB+\cJ_{F})$}
\mor{01}{10}{$\pi_1$}[\atright,\solidarrow]
\mor{01}{21}{$i$}[\atleft, \injectionarrow]
\mor{21}{10}{$\pi_2$}
\enddc
\end{equation}
where the maps $\pi_i$ are defined by $\bfT\mapsto \bfT$. We denote $\cX:=\spec \CC[u,\bfT,\bfx]/(\cJ_\cB+\cJ_{F})$ and $\cB:=\spec \CC[\bfT]/\cJ_\cB$ and thus we have the flat map $\pi_2:\cX\to \cB$ that we are going to analyse in more detail in the upcoming sections.

\begin{theorem}
The above diagram is a deformation of $(X,\partial X)$.
\end{theorem}
\begin{proof}
The fibers over $0$ are $\pi_1^{-1}(0)\cong \partial X$ and $\pi_2^{-1}(0)\cong X$ and we have already proved that $\pi_2$ is flat. Thus also $\pi_1$ is flat (see e.g.\ \cite[Lemma 3.10]{pet2}). Note that we could also prove that $\pi_1$ is flat directly by lifting the relations (the computations are the same as for proving that $\pi_2$ is flat modulo $u$).
\end{proof}

\section{The Kodaira--Spencer map}\label{kod spe sec}

In this section we are going to prove the following theorem.

\begin{theorem}\label{th kodspe}
The Kodaira--Spencer map $T_0\cB\to \bigoplus_{k\in \NN}T^1_{(X,\partial X)}(-km)\subset T^1_{(X,\partial X)}$ of the deformation \eqref{eq def diagram} is bijective and  the Kodaira--Spencer map $T_0\cB\to \bigoplus_{k\in \NN}T^1_{X}(-km)\subset T^1_{X}$ of the map $\pi_2$ in \eqref{eq def diagram} is surjective.
\end{theorem}

\subsection{The tangent space $T^1_{\paar}$ and the obstruction space $T^2_{\paar}$}\label{the tang and obs}

Since $\partial X\hookrightarrow X$ is a regular embedding we 
can use results in \cite{cili} to describe $T^1_{\paar}$.
Let us denote $A=\CC[u,\bfx]/\cI_S$ and thus $X=\spec \CC[u,\bfx]/\cI_S=\spec A$. We know that $\partial X=\spec \CC[u,\bfx]/(\cI_S,u)=\spec A'$ with $A':=A/(u)$ and $\partial X\hookrightarrow X \hookrightarrow \CC^{r+1}$. 

We have the following exact sequence (see e.g.\ \cite[Equation 11]{cili}):
\begin{equation}\label{eq varphi1}
0\to T_{\partial X}\to T_{X|\partial X}\xrightarrow{\varphi} N_{\partial X|X}\xrightarrow{\varphi_1} T^1_{(X,\partial X)}\to T^1_X\to H^1(\cN_{\partial X|X})\to \cdots
\end{equation}
where $T_{\partial X}=\Der_\CC(A',A')$ are derivations from $A'=A/(u)$ to $A'$, $T_{X|\partial X}=\Der_\CC(A,A)\otimes A'$
 and $N_{\partial X|X}=\Hom_{A'}((u)/(u)^2,A')$. Recall the set $B$ from \eqref{eq minb}.

\begin{proposition}\label{t1t2 par}
For $r\in B$, we have $\dim_\CC T^1_{\paar}(-r)=1+\dim_\CC T^1_X(-r)$. Moreover, it holds that $T^2_{X}\cong T^2_{\paar}$ and $T^1_{\paar}\cong T^1_X\bigoplus \operatorname{Im}(\varphi_1)$. 
\end{proposition}
\begin{proof}
Note that as $\partial X\hookrightarrow X$ is a regular embedding, $\cN_{\partial X|X}$ is a line bundle on the (affine) $X$. Hence, $H^i(\cN_{\partial X|X})=0$ for $i>0$. Thus, $T^2_{X}\cong T^2_{\paar}$ and $T^1_{\paar}\cong T^1_X\bigoplus \operatorname{Im}(\varphi_1)$. For any $r = R^* - s \in B$, we observe that the element $u \mapsto \chi^s$ in $N_{\partial X \mid X}$ does not lie in the image of $\varphi$, from which it follows that $\dim_\CC T^1_{\paar}(-r) = 1 + \dim_\CC T^1_X(-r)$.
\end{proof}

\begin{remark}\label{rem kon}
Note that any deformation of the (affine) $X$ induces also a deformation of $\partial X$ by looking modulo $u$ (see e.g.\ \cite[Lemma 3.10]{pet2}). 
\end{remark}

\subsection{The dimension of $\cB$}\label{sub sec of b}
Recall that $\cJ_\cB$ is generated by the polynomials $p^{(j)}_{\un{d}}(\bfT)$, appearing in \eqref{eq pol p}. 
In particular, we see that the tangent space $T_0\cB$ of $\cB=\spec \CC[\bfT]/\cJ_\cB\subset \spec \CC[\bfT]$ at $0$ is  
\begin{equation}\label{eq tang at b}
\Big\{(T_{11},\dots,T_{1l_1},\dots,T_{n_G1},\dots,T_{n_Gl_{n_G}})\in \CC^{\sum_{i=1}^{n_G}l_i}~|~\sum_{d^i;l_i\geq j}\frac{\delta_\eps(d^i)}{l_i}d^iT_{ij}=0,\text{ for }j\in \NN\text{ and }2\text{-face }\eps\text{ in }G\Big\}.
\end{equation}
Indeed, $f_{\bfu\to \bfT}(p_{\un{d}}(\bfu))$ is modulo $(\bfT)^2$ equal to $\sum_{j=1}^{g_{\un{d}}}\sum_{d^i;l_i\geq j}\frac{\delta_\eps(d^i)}{l_i}\lan d^i,c\ran \bfx_m^ju^{g_{\un{d}}-j}T_{ij}$. 

Recall by Remark \ref{rem depg} that $\cB$ depends on $m$.
\begin{proposition}\label{prop 32}
For $m\in B\subset \tM$ we have 
$\dim_\CC T_0\cB=\dim_\CC \bigoplus_{k\in \NN}T^1_{(X,\partial X)}(-km)$.
\end{proposition}
\begin{proof}
For $j\in \NN$ we denote 
\begin{equation}\label{eq t0b}
T_0\cB(j):=\{(T_{1j},\dots,T_{n_Gj})\in \CC^{n_G}~|~ T_{ij}=0\text{ if }l_i<j,~\sum_{d^i;l_i\geq j}\frac{\delta_\eps(d^i)}{l_i}d^iT_{ij}=0,\text{ for each }2\text{-face }\eps\text{ in }G\},
\end{equation}
where $G=(R^*=1)\cap (m=1)\cap \sigma$ as before. By a well-studied description of $T^1_X$ (see \cite[Theorem~4.1]{ka-flip}, where it is described precisely in terms of the vector space appearing in \eqref{eq t0b}, denoted by $V'_\CC(jm)$ in that paper) and Proposition~\ref{t1t2 par}, we immediately see
 that for $j\geq 2$ we have $$\dim_\CC T_0\cB(j)=\dim_\CC T^1_X(-jm)=\dim_\CC T^1_{\paar}(-jm)$$ and 
 \begin{equation}\label{eq t0cb}
 \dim_\CC T_0\cB(1)=\dim_\CC T^1_X(-m)+1=\dim_\CC T^1_{\paar}(-m).
 \end{equation}
  From the equation \eqref{eq tang at b} we have $T_0\cB=\bigoplus_{j\in \NN}T_0\cB(j)$, from which the proof follows.
\end{proof}

\begin{example}\label{ex 45}
In our Example \ref{ex house pop} we have 
$$\dim_\CC T^1_{\paar}(-R^*)=3,~~~\dim_\CC T^1_{\paar}(-2R^*)=1,~~~\dim_\CC T^1_{\paar}(-kR^*)=0,~~~\text{for }k\geq 3.$$
\end{example}

\begin{remark}
Note that with Proposition \ref{prop 32} we see that our choice of $u_i$ in \eqref{eq ui new 2} was natural since we obtain the right dimension of the tangent space and we also see why $m$ needs to lie in $B$ since otherwise we cannot apply formulas for computing $T^1_X$ and Proposition \ref{t1t2 par} to get the right dimension of the tangent space. For example, we could also define $f_{\bfu\to \bfT}(u_i)=u^{l_i}+\bfx_mT_{il_i}$ but if $l_i>1$ we only get a strict subset of $\dim_\CC \bigoplus_{k\in \NN}T^1_{(X,\partial X)}(-km)$ for the tangent space. If $l_i=1$ for all $i$ and $m=R^*$, then $f_{\bfu\to \bfT}(u_i)=u+T_{i1}$ so we are in the case of \cite{alt} by Altmann. Here we see that it is more natural to consider deformations of $\paar$ due to \eqref{eq t0cb}, which was also mentioned in the introduction. Now that we naturally obtain a flat family with the right dimension of its tangent space, the goal is to prove that this family is in fact miniversal in degrees $-km$, $k\in \NN$, by proving bijectivity of the Kodaira--Spencer map and surjectivity of the obstruction map.
\end{remark}

\subsection{The Kodaira--Spencer map}

We refer the reader to \cite[Section 10]{de jong} for the definition and the construction  of a Kodaira--Spencer map.
In the following we will construct  the Kodaira--Spencer map $T_0\cB\to \bigoplus_{k\in \NN}T^1_{X}(-km)$ of the map $\pi_2$, cf.\ \eqref{eq def diagram}.
As before let us write $A:=\CC[u,x_1,\dots,x_r]/\cI_S$. The following exact sequence is well known:
\begin{equation}\label{eq dercaa}
0\to \Der_\CC(A,A)\to A^{r+1}\xrightarrow{\xi} \Hom_{A}(\faktor{\cI_S}{\cI^2_S},A)\xrightarrow{} T^1_X=\operatorname{coker}(\xi)\to 0,
\end{equation}
where the map $\xi$ maps an element $(h,h_1,\dots,h_{r})\in A^{r+1}$ to 
$$
\bar{f}\mapsto h\frac{\partial f}{\partial u}+\sum_{i=1}^{r}h_i\frac{\partial f}{\partial x_i}\in \Hom_{A}(\faktor{\cI_S}{\cI^2_S},A).
$$
Computing $F_\bfk(u,\bfT,\bfx)$ from \eqref{eq fbfk flat 2} modulo $(\bfT)^2$ gives us  
$$F_\bfk(u,\bfT,\bfx)=f_\bfk(u,\bfx)+\sum_{i=1}^{n_G}\sum_{j=1}^{l_i}z_i(\bfk)T_{ij}\bfx^j_m\cdot \bfx^{\bo(\bfk)}u^{\eta_P(\bfk)-j}\in \CC[u,\bfT,\bfx]/(\bfT)^2.$$
 Thus the Kodaira--Spencer map of the flat map $\pi_2$ is given by 
$T_0\cB\xrightarrow{K_2}  T^1_X,$
where $$K_2(\bfT)=\Big(f_\bfk\mapsto \sum_{i=1}^{n_G}\sum_{j=1}^{l_i}z_i(\bfk)T_{ij}\bfx^j_m\cdot \bfx^{\bo(\bfk)}u^{\eta_P(\bfk)-j}\in A\Big)\in T^1_X$$ and we look on $T^1_X$ as a cokernel of the map $\xi$, cf.\ \eqref{eq dercaa}.
Now the image of $\xi$ in degree $-m$ is one-dimensional  and the image of $\xi$ in degrees $-km$, for $k\geq 2$, equals zero.
By restricting the codomain to $\bigoplus_{k\in \NN}T^1_X(-km)\subset T^1_X$ we see that $K_2$ is surjective and has one-dimensional kernel. This one-dimensional kernel of $T_0\cB$ induce one-parameter deformation of $\paar$ that non-trivially deforms $\partial X$ and trivially deforms $X$.  
Its image under the Kodaira--Spencer map $K:T_0\cB\to \bigoplus_{k\in \NN}T^1_{(X,\partial X)}(-km)\subset T^1_{(X,\partial X)}$ equals $\operatorname{im(\varphi_1)}$, cf.\ Proposition \ref{t1t2 par}.
Using Proposition \ref{prop 32} we thus proved Theorem \ref{th kodspe}.

\section{The miniversal deformation}\label{sec obs}

\subsection{The obstruction map}
In the following we are going to show that the map $\pi_2$ (appearing in the deformation diagram \eqref{eq def diagram}) is surjective. This implies that the deformation \eqref{eq def diagram} of $\paar$ is maximal with the prescribed tangent space $\bigoplus_{k\in \NN}T^1_{\paar}(-km)$, i.e.\ we can not extend it to a deformation of $\paar$ with a larger base space (by keeping the tangent space fixed). In this case we also say the deformation diagram \eqref{eq def diagram} is \emph{miniversal in degrees $-km$} for all $k\in \NN$. 

\begin{remark}
The justification for using the term miniversal in degrees $-km$ is that if the miniversal deformation of $\paar$ exists (which is the case if $X$ is three-dimensional, see \cite{mojcor}), then we get the miniversal deformation in degrees $-km$, $k\in \NN$, by restricting it to only those variables coming from $\bigoplus_{k\in \NN}T^1_{\paar}(-km)$. In particular, we obtain the miniversal deformation in degree $-km$ (for some $k \in \NN$) by restricting to only those variables coming from $T^1_{(X, \partial X)}(-km)$.
\end{remark}

We briefly recall the definition of the obstruction map from \cite[Section 10]{de jong}. 
 Let $\cR$ be the module of linear relations between the equations $f_\bfk\in \cI_\kS$ defining $X=\spec A$.
 The module $\cR$ contains the submodule $\cR_0$ of the so-called Koszul relations and we have 
\begin{equation}\label{def t2 prva}
T^2_X:=\frac{\Hom(\cR/\cR_0,\,A)}{\Hom(\bigoplus_{\bfk\in \NN^r}\CC[\bfx,u]f_\bfk,\,A)}.
\end{equation}

Since we will no longer use the total number of edges of $P$, 
we denote by $n:=n_G$ the total number of edges of $G$ for simplicity.
From \eqref{pund} recall  $p_{\un{d}}(\bfu)\in \cI_{\ktT(G)}\subset \CC[\bfu]$ with 
\begin{equation}\label{eq undund}
\un{d}=\un{d}_{c,\eps}=\Big(\delta_\eps(d^1)\lan d^1,c\ran,\dots,\delta_\eps(d^{n_G})\lan d^{n_G},c\ran\Big)\in \cT^*_\ZZ(G) \cap \cT(G)^\perp
\end{equation}
for some $c\in M_\QQ$ and some $2$-face $\eps$ in $G$.
Recall also that $\cJ_\cB$ is generated by $p^{(j)}_{\un{d}}(\bfT)$, cf.\ \eqref{eq pol p}.
We consider the ideal
 $$\widetilde{\cJ_\cB}:=\cJ_\cB\cdot (\bfT)+\cJ'_\cB\CC[\bfT]\subset \CC[\bfT],$$
where $$\cJ'_\cB:=\big(p^{(k)}_{\un{d}}(\bfT)~|~p^{(k)}_{\un{d}}(\bfT)\text{ contains a monomial $aT_{ij}$ for some $a\in \CC\setminus \{0\}$ and $i,j\in \NN$}\big)$$ denotes the ideal generated by only those $p^{(k)}_{\un{d}}(\bfT)$ that contain a monomial $aT_{ij}$. 
We define a $\ZZ$-graded vector space $W:=\cJ_\cB/\widetilde{\cJ_\cB}$ with $W=\bigoplus_{k\in \NN}W_k$. 

\begin{remark}
If all edges appearing in $G=P$ have  lattice length $1$, then $\cJ'_\cB$ is generated by degree $R^*$ elements in $\cJ_\cB$.  
\end{remark}

\begin{example}
In our example we computed the generators of the ideal $\cJ_\cB$ in \eqref{eq houe ex equ 2}. The ideal $\cJ'_\cB$ is in this case equal to $(T_{11}+T_{51}-T_{21}-T_{31},T_{11}T_{51}+T_{52}-T_{32}-T_{21}T_{31})$.\demo
\end{example}

From \eqref{eq obs flat} recall $\widetilde{R}_{\bfa,\bfk}$, 
where $p^{(j)}_{\bfa,\bfk}(\bfT)\in \cJ_\cB$ are defined in \eqref{eq pol p}.  Recall $R_{\bfa,\bfk}$ from Lemma \ref{lem 29}. 
Let  $o\in \Hom(\cR/\cR_0,A\otimes W)$ be defined by
\begin{equation}\label{eq obs eq}
o(R_{\bfa,\bfk})=\bfx^{\bo(\bfa+\bfk)}u^{\eta_P(\bfa+\bfk)-\eta_G(\bfa+\bfk)}\Big(\sum_{k=1}^{\eta_G(\bfa+\bfk)}\bfx_m^ku^{{\eta_G(\bfa+\bfk)}-k}p^{(k)}_{\bfa,\bfk}(\bfT)\Big)\in A\otimes W.
\end{equation}
It holds that 
$$o\in \Hom(\cR/\cR_0,A\otimes W)=\Hom(\cR/\cR_0,A)\otimes W=T^2_X\otimes W=\Hom((T^2_X)^*,W)$$
and $o:(T^2_X)^*\to W$ is called the \emph{obstruction map} of the map $\pi_2$.

\subsection{Toric description of the obstruction map}

The following definitions already appeared in \cite[Section 6]{alt}. 
Recall the Hilbert basis $E$ of $S=\sigma^\vee\cap \tM$ from the equation \eqref{eg hilbbas} and for $R\in \tM$ we consider
\[
  E_{a_i}^R:=E_i^R:=\{e\in E~|~\lan a^i,e \ran<\lan a^i,R \ran\}.
\]
For a subface $\tau$ of $\sigma$ (denoted $\tau\leq \sigma$) let $E^R_\tau:=\bigcap_{a^i\in \tau}E^R_{i}$. The $\ZZ$-module of all linear relations among elements in $E^R_\tau$ we denote by $L(E^R_\tau)$.  

\begin{proposition}\label{prop }
\begin{equation}\label{eq t2prva}
  T^2_X(-R)^*\cong \Bigg(\frac{\ker\Big( \bigoplus_{i}L_\CC(E^R_i)\to L_\CC(E)\Big)}{\operatorname{image}\Big(\bigoplus_{\lan a^i,a^k\ran\leq \sigma}L_\CC(E_i^R\cap E_k^R)\to \bigoplus_{i}L_\CC(E^R_i)\Big)} \Bigg).
\end{equation}
\end{proposition}
\begin{proof}
See \cite[Propositions 5.4, 5.5]{klaus}.
\end{proof}

By $v_*$ we denote the vertex $0$ of $G=P\cap (m=1)\subset P$ for some $m\in B$, cf.\ \eqref{eq minb}.
From now on we will work only with $G$ (instead of $P$), and for simplicity 
we write $n:=n_G$ for the number of edges of $G$ and $v(c):=v_G(c)$ for $c\in M$.
The following we recall from \cite[Definition 3.6]{m}.  

\begin{definition}\label{def path}
Let $E_1,\dots,E_{n}$ be the edges of $G$, oriented by direction vectors $d^1,\dots,d^{n}$.  
For a path $p=p_{w^1 \leadsto w^k}$ along the edges of $G$ we define its \emph{edge-count vector}
\[
\#(p) := (\nu_1(p),\dots,\nu_{n}(p)) \in \ZZ^{n},
\]
where $\nu_i(p)$ is the signed number of times the path $p$ traverses the edge $E_i$ 
(with orientation given by $d^i$). Thus $\#(p)$ records, for each edge, how often 
and in which direction the path passes through it.

For $a,c\in M$ we set
\[
  \underline\lambda(a) := \#\big(p_{v_* \leadsto v(a)}\big), \qquad
  \underline\mu^c(a) := \#\big(p_{v(a) \leadsto v(c)}\big),
\]
where in the second case the path is chosen such that
$\mu^c_i(a)\,\langle c, d^i\rangle \leq 0$ for all $d^i$. Finally we define
\[
  \underline\lambda^c(a) := \underline\lambda(a)+\underline\mu^c(a).
\]
\end{definition}

Recall $p_{\un{d}}(\bfu)$ from \eqref{pund} and $p^{(k)}_{\un{d}}(\bfT)$ from \eqref{eq pol p}. 
As in \eqref{eq undund}, for any 
$\underline{\mu}=(\mu_1,\dots,\mu_n)\in \ZZ^n$ satisfying 
$
\sum_{i=1}^n \mu_i d^i = 0,
$
we define
\begin{equation}\label{eq unkd}
\un{d}(\un{\mu},c):=(\lan \mu_1d^1,c\ran,\dots,\lan \mu_nd^n,c\ran)\in\cT^*_\ZZ(G) \cap \cT(G)^\perp\text,~~~p(\un{\mu},c):=p_{\un{d}(\un{\mu},c)}(\bfu), ~~~p^{(k)}(\un{\mu},c):=p^{(k)}_{\un{d}(\un{\mu},c)}(\bfT),
\end{equation}
where $p^{(k)}_{\un{d}(\un{\mu},c)}(\bfT)$ is homogeneous of degree $km$, cf.\ \eqref{eq pol p}.
We define the map
\[
  \psi_i^{(k)}:L_\CC(E^{km}_{a^i})\to W_k,
\]
\[
  \un{q}\mapsto \sum_{j=1}^rq_jp^{(k)}(\un{\lam}^{c_j}(v^i)-\un{\lam}(v(c_j)),c_j).
\]

\begin{lemma}\label{bil lem}
$p^{(k)}(\un{\mu},c)\in W_k$ is a bilinear map:
$$p^{(k)}(\un{\mu_1}+\un{\mu_2},c)=p^{(k)}(\un{\mu_1},c)+p^{(k)}(\un{\mu_2},c)\in W_k\text{  and  }p^{(k)}(\un{\mu},c_1+c_2)=p^{(k)}(\un{\mu},c_1)+p^{(k)}(\un{\mu},c_2)\in W_k.$$ 
\end{lemma}
\begin{proof}
Straightforward computation shows that
$$\frac{1}{2}p_{\un{d}}(\bfu)\big(\prod_{i=1}^n u_i^\frac{d_i^+}{l_i}+\prod_{i=1}^nu^{\frac{d_i^-}{l_i}}_i\big)+\frac{1}{2}p_{\un{e}}(\bfu)\big(\prod_{i=1}^n u_i^\frac{e_i^+}{l_i}+\prod_{i=1}^nu^{\frac{e_i^-}{l_i}}_i\big)=\prod_{i=1}^n u_i^\frac{d_i^+}{l_i}\prod_{i=1}^n u_i^\frac{e_i^+}{l_i}-\prod_{i=1}^n u_i^\frac{d_i^-}{l_i}\prod_{i=1}^n u_i^\frac{e_i^-}{l_i}=$$
$$p_{\un{d}+\un{e}}(\bfu)\prod_{i\in S_1}u_i^{\frac{e_i^-}{l_i}}\prod_{i\in S_2}u_i^{\frac{d_i^-}{l_i}}\prod_{i\in S_3}u_i^{\frac{d_i^+}{l_i}}\prod_{i\in S_4}u_i^{\frac{e_i^+}{l_i}},$$
where 
$$S_1=\{i\in \{1,\dots,n\}~|~d_i>0,~e_i<0,~d_i+e_i>0\},~~S_2=\{i\in \{1,\dots,n\}~|~d_i<0,~e_i>0,~d_i+e_i>0\},$$
$$S_3=\{i\in \{1,\dots,n\}~|~d_i>0,~e_i<0,~d_i+e_i<0\},~~S_4=\{i\in \{1,\dots,n\}~|~d_i<0,~e_i>0,~d_i+e_i<0\}.$$
Now our claim easily follows because $W$ is a quotient space $W=\cJ_\cB/\widetilde{\cJ_\cB}$.
\end{proof}

\begin{proposition}\label{th obs map}
$\psi_i^{(k)}$ induce the linear map $\psi^{(k)}: T^2_X(-km)^*\to W_k$ and the map $$\psi=\sum_{k\in \NN}\psi^{(k)}: \bigoplus_{k\in \NN}T^2_X(-km)^*\to W$$ is the obstruction map of the flat map $\pi_2$.
\end{proposition}
\begin{proof}
The idea of the first part of the proof is similar to \cite[Lemma 7.7]{alt}.
Let $\rho^{ij}$ denote the path consisting of the single edge running from $v^i$ to $v^j$. For $\un{q}\in L(E^{km}_{a^i}\cap E^{km}_{a^j})$ we see by Lemma \ref{bil lem} that
\[
  \psi_i^{(k)}(\un{q})-\psi_j^{(k)}(\un{q})=\sum_{l=1}^rq_lp^{(k)}(\un{\lam}(a^i)-\un{\lam}(a^j)+\rho^{ij},c_l)+\sum_{l=1}^r q_lp^{(k)}(\un{\mu}^{c_l}(a^i)-\un{\mu}^{c_l}(a^j)-\rho^{ij},c_l).
\]
We want to show that the above expression is equal to $0$. The first sum is zero by Lemma \ref{bil lem} using $\sum_{l=1}^rq_lc_l=0$. For the second sum, we observe that for every $\un{q} \in L(E^{kR^*}_{a^i} \cap E^{kR^*}_{a^j})$, the following holds: if $q_l \ne 0$, then $(c_l; \eta(c_l)) \in E$ satisfies
$
\langle (c_l; \eta(c_l)), a^i \rangle < \langle kR^*, a^i \rangle = k.
$
Using the identity $a^i = (v_i; 1)$, this implies
$
\langle c_l, v_i \rangle - \langle c_l, v(c_l) \rangle < k,
$
and similarly,
$
\langle c_l, v_j \rangle - \langle c_l, v(c_l) \rangle < k.
$
From this, it follows that the degree of 
$
p(\un{\mu}^{c_l}(a^i) - \un{\mu}^{c_l}(a^j) - \rho^{ij}, c_l)
$
is $zm$ for some $z < k$. Therefore,
$
p^{(k)}(\un{\mu}^{c_l}(a^i) - \un{\mu}^{c_l}(a^j) - \rho^{ij}, c_l) = 0,
$
which concludes the proof that $\psi_i^{(k)}$ induce the linear map
$
\psi^{(k)}: T^2_X(-km)^* \longrightarrow W_k.
$

The proof that $\psi$ is the obstruction map is similar to \cite[Proposition 7.5]{m} or \cite[Proposition 7.8]{alt} so we just highlight the main idea: using \cite[Theorem 3.5]{alt2} we can find an element of
$
\operatorname{Hom}(\cR/\cR_0, A\otimes W_k)
$
representing $\psi^{(k)}$. It sends the relation $R_{\bfa,\bfk}$ to
\begin{equation} \label{eq:psi_n_relation}
\psi^{(k)}(R_{\bfa,\bfk}) =
\begin{cases}
\left( \psi^{(n)}_{v(c_\bfk)}(\bfk - \partial(\bfk)) - \psi^{(k)}_{v(c_\bfa + c_\bfk)}(\bfk - \partial(\bfk)) \right) x^{\bfa + \bfk - km}, & \text{if } \eta_P(\bfa+\bfk) \geq k, \\
0, & \text{otherwise}.
\end{cases}
\end{equation}
This element induce the same element in $\operatorname{Hom}(\cR/\cR_0, A\otimes W_k)$ as $o$ from \eqref{eq obs eq}.
\end{proof}

\subsection{Surjectivity of the obstruction map}

In this section we prove the surjectivity of the obstruction map $\psi$. The idea of the proof is new, with the previous techniques we were not able to obtain the surjectivity of the obstruction map even in the single Gorenstein degree $-R^*$, if $P$ has at least one edge of lattice length $\geq 2$, cf. \cite[Example 6.5, Remark 7.9]{m}.

Let $\epsilon$ be a $2$-face in $G=(m=1)\cap (R^*=1)\cap \sigma\subset P$ with cyclically ordered vertices $v^1, \dots, v^n$, where we set $v^{n+1} := v^1$. Let $a^i = (v^i; 1)$ and define $d^i := v^{i+1} - v^i$. Then we have
$\sum_{i=1}^n d^i = 0.$ For $R\in \tM$ we denote $$K_{a_i}^R:=K_i^R:=\{r\in S~|~\lan a^i,r \ran<\lan a^i,R \ran\}$$ and $K^R_{i,i+1}:=K^R_{a^i}\cap K^R_{a^{i+1}}$.

Let $\varphi_\eps:=\sum_{k\in \NN}\varphi^{(k)}_\eps$, where 
$$\varphi^{(k)}_\eps:\Big(\bigcap_i(\span_\ZZ K^{km}_{i,i+1})/\span_\ZZ(\bigcap_iK_{i,i+1}^{km})\Big) \to W_k$$
$$(c;m)\in \tM\mapsto p^{(k)}_\eps(c):=p^{(k)}(\un{1}_\eps,c).$$
Note that we have already oriented $\eps$, which is a $2$-face, and thus we can simply take $\un{\mu} := \un{1}_\eps$, meaning that it takes the value $1$ on every edge of $\eps$ and $0$ elsewhere.
 Let us check that the map $\varphi^{(k)}_\eps$  is well defined: we need to show that 
\begin{equation}\label{eq phik}
\varphi^{(k)}_\eps(c)=0\in W_k~~~\text{for }c\in \bigcap_iK_{i,i+1}^{km}.
\end{equation}
For $c\in M$ we denote 
\begin{equation}\label{eq dc}
d(c):=\max\{\lan v^i,c\ran~|~i=1,\dots,n\}-\min\{\lan v^i,c\ran~|~i=1,\dots,n\}.
\end{equation}
We immediately see that the degree of the homogeneous polynomial $$p_{\un{d}}(\bfu)\in \cI_{\ktT(P)},\quad\un{d}=(\lan d^1,c\ran,\dots,\lan d^n,c\ran),$$ is equal to $d(c)m$, cf.\ \eqref{pund}. Thus \eqref{eq phik} follows from the following lemma.
\begin{lemma}\label{lem dc}
There exists $z\in \ZZ$ such that $(c;z)\in \bigcap_iK_{i,i+1}^{km}$ if and only if $d(c)\leq k-1$.
\end{lemma}
\begin{proof}
It follows immediately by definitions: note that $a^i=(v^i;1)$ and that $r\in \bigcap_iK_{i,i+1}^{km}=\bigcap_iK^{km}_{a^i}$ if and only if $0\leq \lan a^i,r\ran\leq k-1$ for every $i=1,\dots,n$.
\end{proof}
\begin{corollary}
The map $\varphi_\eps$ is well defined.
\end{corollary}

\begin{remark}
We will see from the proof of Theorem \ref{prop prop 65} that the maps $\varphi_\eps$ play a crucial role in proving the surjectivity of the obstruction map $\psi$. Moreover, if $X_P$ is three-dimensional (with $P=\eps$), then $\psi$ is the $\CC$-linear extension of $\varphi_\eps$.
\end{remark}

\begin{lemma}\label{edge lem pom}
For an edge $d^i=v^{i+1}-v^i$ it holds that
\begin{equation}\label{eq 1}
c\in (d^i)^\perp
\end{equation}
if and only if there exists $z\in \NN$ such that 
\begin{equation}\label{eq 2}
(c;z)\in (a^i)^{\perp}\cap (a^{i+1})^\perp
\end{equation}
\end{lemma}
\begin{proof}
Recall that $a^i=(v^i,1)\in \tN$ and thus \eqref{eq 2} follows from \eqref{eq 1} by picking $z:=-\lan c,v^i\ran=-\lan c,v^{i+1}\ran$. From \eqref{eq 2} it follows that $\lan c,v^i\ran=\lan c,v^{i+1}\ran$, from which \eqref{eq 1} follows.
\end{proof}

\begin{theorem}\label{prop prop 65}
  The map $\psi:\bigoplus_{k\in \NN}T^2_X(-km)^*\to \bigoplus_{k\in \NN}W_k$ is surjective.
\end{theorem}
\begin{proof}
Recall the description of $T^2_X$ from \eqref{eq t2prva}. We need to show that $p^{(k)}_{\un{d}}(\bfT)$ are in the image of $\psi^{(k)}$ for $\un{d}=\un{d}_{c,\eps}$ (for every two face $\eps$), cf.\ \eqref{eq undund}. 

Let us fix a $2$-face $\eps$ (with vertices $v^i$, $i=1,\dots,n$).
 Starting from $(c;z)\in \bigcap_{i=1}^n(\span_\ZZ K^{km}_{i,i+1})$ we obtain the corresponding element 
$$L(c)\in \ker\Big( \bigoplus_{i=1}^nL_\CC(E^{km}_i)\to L_\CC(E)\Big)$$
as follows: we can write
\begin{equation}\label{eq cqij}
c=\sum_{j=1}^rq_{i,j}c_j+q_i(\un{0},1),
\end{equation}
where $q_{i,j}\ne 0$ implies that $(c_j;\eta(c_j))\in E^{km}_{d^i}:=E^{km}_{a^i}\cap E^{km}_{a^{i+1}}$. Let
$$L(c)_i:=\sum_j(q_{i,j}-q_{i-1,j})(c_j;\eta(c_j))+(q_i-q_{i-1})(\un{0},1)=0$$
be an element in $L(E^{km}_i)$,
which defines $L(c):=\sum_iL(c)_i\in \bigoplus_{i=1}^nL(E^{km}_i)$.

To show that $\psi^{(k)}(L(c)) = \varphi^{(k)}_\eps(c) = p^{(k)}_\eps(c)$, we need to verify that
$$
\sum_{i=1}^n \sum_{j=1}^r (q_{i,j} - q_{i-1,j}) \, p^{(k)}\left( \un{\lambda}^{c_j}(v^i) - \un{\lambda}(v(c_j)), c_j \right) = p^{(k)}_\eps(c).
$$
Using Lemma \ref{bil lem} and the path $\rho^{ij}$ from the proof of Proposition \ref{th obs map}, this is a straightforward computation, similarly as in \cite[Section 7.9(iii)]{alt}.

Thus we show that for any $c \in \bigcap_{i=1}^n (\span_\ZZ K^{km}_{i,i+1})$, there is $p^{(k)}_\eps(c) \in W_k$.  
To finish the proof, it is enough to show that if
\begin{equation}\label{eq statem}
\text{for each } z \in \ZZ \text{ it holds that } (c;z) \not\in \bigcap_{i=1}^n \span_\ZZ K^{km}_{i,i+1},
\end{equation}
then $p^{(k)}_\eps(c) = 0 \in W_k$.
 For $k\geq 2$ we immediately see that 
\begin{equation}\label{eq spanzz}
\span_\ZZ K^{km}_{i,i+1}\cong \left\{
\begin{array}{ll}
\span_\ZZ \Big( \tM\cap (a^i)^{\perp}\cap (a^{i+1})^{\perp},m \Big)&\text{ if }\ell(d^i)\geq k \\
\tM & \text{ if }\ell(d^i)< k.
\end{array}
\right.
\end{equation}
For $c\ne 0$ we see by Lemma \ref{edge lem pom} and \eqref{eq spanzz} that if \eqref{eq statem} holds, then $\lan c,d^{i}\ran\ne 0$ for some $d^i$ with $\ell(d^i)\geq k$, from which it follows that  $p^{(k)}(c)=0\in W_k$ since the coefficient in front of $T_{ik}$ in $p^{(k)}(c)$ is non-zero. 
\end{proof}

Thus we proved the following.

\begin{theorem}\label{th miniversal pair}
 The deformation diagram \eqref{eq def diagram} is the miniversal deformation of the pair $\paar$ in degrees $-km$, $k\in \NN$. Moreover, the flat map $$\pi_2: \spec \CC[u,\bfT,\bfx]/(\cJ_\cB+\cJ_{\ktS})\to \spec \CC[\bfT]/\cJ_\cB,$$
is a versal deformation of $X$ in degrees $-km$, $k\in \NN$. 
 \end{theorem}
\begin{proof}
This follows from Theorem~\ref{prop prop 65} (surjectivity of the obstruction map $\psi$) and Theorem~\ref{th kodspe} (bijectivity of the Kodaira--Spencer map in the case of deforming $\paar$, and surjectivity of the Kodaira--Spencer map in the case of deforming $X$); see, e.g.,~\cite[Corollary 10.3.20]{de jong}.
\end{proof}

\begin{remark}\label{proj rem}
Let $P$ be a reflexive polytope and $$m\in\tM_0=\{(c;0)\in \tM~|~c\in M\},$$
i.e.,  $\tM_0$ consists of those $m$ that their projection to the last component is $0$, i.e.\ 
 For such $m$ the deformations in degree $-m$ are called \emph{degree $0$ deformations} and by a comparison theorem of Kleppe \cite{kle} those deformations induce deformations of the toric Gorenstein Fano variety $Y$ associated to the face fan of $P$. Moreover, the tangent space of deformations of $Y$ is isomorphic to $\bigoplus_{m\in \tM_0}T^1_X(-m)$, where $X=X_P$ is the affine cone of $Y$, the obstruction space of deformations of $Y$ is also isomorphic to $\bigoplus_{m\in \tM_0}T^2_X(-m)$. 
\end{remark}

\begin{corollary}\label{cor proj rem}
   A versal deformation of $X$ in degrees $-km$, for all $k \in \NN$ and $m \in \tM_0$, induces a versal deformation of $Y$ in the same degrees.
\end{corollary}

\section{Irreducible components of the reduced miniversal space}\label{sec irr compo}
In this section we show that irreducible components of our constructed reduced miniversal space of $X_P$ in degrees $-km$, $k\in \NN$, are in one to one correspondence with maximal Minkowski decompositions of $G=P\cap (m=1)$. 
This is a generalization of Altmann's result in \cite{alt}, where it was shown that if $P$ has edges of lattice length $1$, then the components of the miniversal space in degree $-R^*$ of $X_P$ are in one to one correspondence with maximal Minkowski decompositions of $P=P\cap (R^*=1).$ In our result it is interesting that the Minkowski decompositions of $P\cap (m=1)$ do not encode the components of miniversal space in degree $-m$ but in fact encode the components of the whole miniversal space in degrees $-km$, for all $k\in \NN$. However, this is not surprising since this happens already in the two dimensional case, which we cover in the following remark.

\begin{remark}
The two dimensional affine Gorenstein toric varieties are $A_n$-singularities given by the equation $xy-z^{n}\subset \CC[x,y,z]$. The polytope $P$ that is defining $X=\spec \CC[x,y,z]/(xy-z^{n})$ is a line segment of lattice length $n$ (say $P=[0,n]$). The miniversal deformation of $\paar$ (or $X$) is very well known since $X$ is a hypersurface: the deformations of $\paar$ (resp.\ $X$) are unobstructed with the dimension of the miniversal base space equal to $n$ (resp.\ $n-1$). Note that this follows also from our construction since, if $P$ is a line segment, we do not have equations of the base space. Moreover, we know that $\dim_\CC T^1_X(-kR^*)=1$ for all $k=2,3,\dots,n$ and it is $0$ for all other lattice degrees (see e.g.\ \cite[Theorem 2.5]{ka-flip} or \cite[Proposition 2.5]{m}). From Proposition \ref{t1t2 par} then follows that $\dim_\CC T^1_X(-kR^*)=1$ for all $k=1,2,\dots,n$ and it is $0$ for all other lattice degrees. Since $P= [0,n]=(R^*=1)\cap \sigma$ we see that we have only one maximal Minkowski decomposition $P=[0,1]+\cdots+[0,1]=n[0,1]$ which correspond to the only  component of the miniversal base space. \demo
\end{remark}

Let $Q_1+\cdots+Q_p$ be a Minkowski decomposition $G=P\cap (m=1)$ (where $Q_k$ are lattice polytopes for $k\in \{1,\dots,p\}$) and 
let $n_{ik}\in \NN$ be the lattice length of the part of the edge $d^i$ that lies in $Q_k$ for $i\in\{1,\dots,n_G\}$, $k\in \{1,\dots,p\}$, i.e.\ 
$\ell(d^i)=l_i=\sum_{k=1}^pn_{ik}$.

Let $\CC[\bfZ]:=\CC[Z_1,\dots,Z_p]$. Recall \eqref{eq pc} and define the map $$h:\spec \CC[\bfZ]\to \spec \CC[\bfu]/\cI_{\ktT(G)}~~~\text{by}~~~u_i\mapsto \prod_{i=1}^{n_G}Z^{n_{ip}}_i.$$
Clearly this map is well defined (by definition of $\cI_{\ktT(G)}$ we immediately see that $h^*(\cI_{\ktT(G)})=0$) and the kernel of
$h^*$ is a prime ideal, because the image of $h^*$ is an integral domain.

\begin{proposition}\label{prop red 1}
The irreducible components of the reduced space of $\spec \CC[\bfu]/\cI_{\ktT(G)}$ are in one to one correspondence with maximal Minkowski decomposition of $G=Q_1+\cdots+Q_p$. Intersection of components are obtained by the finest Minkowski decompositions of $G$ that are coarser than all the maximal ones involved. This correspondence is given by the map $h$.
\end{proposition}
\begin{proof}
It follows immediately from \cite[Section 2 and 3]{alt} just observe that instead of one variable $u_i$, which correspond to the edge $d^i$ with lattice length $l_i$, we have $l_i$ variables $w_{i1},w_{i2},\dots,w_{il_i}$ that correspond to a line segment $\frac{d^i}{\ell_i}$ (of lattice length $1$) and $u_i=w_{i1}w_{i2}\cdots w_{il_i}$.
\end{proof}

Moreover, we define the map $$g:\spec \CC[\bfZ]\to \cB=\spec \CC[\bfT]/\cJ_{\cB}~~~~~\text{by}$$ 
sending $T_{ij}$ to the degree $j$ part of the polynomial $$(1+Z_1)^{n_{i1}}\cdots (1+Z_p)^{n_{ip}}\in \CC[Z_1,\dots,Z_p].$$ Writing explicitly 
$$
g^*(T_{ij})=
\sum_{r_{ik}\in \NN;\sum_{k=1}^pr_{ik}=j}{n_{i1}\choose r_{i1}}\cdots {n_{ip}\choose r_{ip}}Z_1^{r_{i1}}\cdots Z_p^{r_{ip}}.
$$
From the construction of our miniversal deformation we see that $g$ is well defined. The kernel of
$g^*$ is a prime ideal, because the image of $g^*$ is an integral domain.

\begin{proposition}
The irreducible components of the reduced space of $\cB=\spec \CC[\bfT]/\cJ_\cB$ are in one to one correspondence with maximal Minkowski decomposition of $G=Q_1+\cdots+Q_p$. Intersection of components are obtained by the finest Minkowski decompositions of $G$ that are coarser than all the maximal ones involved. This correspondence is given by the map $g$.
\end{proposition}
\begin{proof}
This follows from Proposition \ref{prop red 1} using the following:  $\cJ_\cB\subset \CC[\bfT]\subset \CC[\bfx_m,u,\bfT]$ is the smallest ideal that contains $f_{\bfu\to \bfT}\left(\cI_{\ktT(G)}\right)\subset \CC[\bfx_m,u,\bfT]$ and is generated by polynomials from $\CC[\bfT]$.
\end{proof}

\begin{corollary}\label{lat cor}
Irreducible components of our constructed reduced miniversal space of $X_P$ in degrees $-km$, $k\in \NN$, are in one to one correspondence with maximal Minkowski decompositions of $G=P\cap (m=1)$, where the summands are lattice polytopes.
\end{corollary}

\begin{remark}
If $X$ is three-dimensional, then 
Corollary \ref{lat cor} is the first step towards describing all reduced irreducible components (see \cite[Conjecture A]{mojcor} for a conjecture on smoothing components) since in particular it says that all reduced irreducible components in degrees $-kR^*$, $k\in \NN$ are in one to one correspondence with maximal Minkowski decomposition of the polygon $P$ (defining $X$) into lattice polytopes.
\end{remark}

\begin{example}\label{ex house mink}
Recall the ideal $\cJ_\cB$ of $\cB$ from \eqref{eq houe ex equ 2} in Example \ref{ex house}.  We can compute that $\cB$ has two irreducible components given by the following two ideals:
$$I_1=(T_{11}-T_{21},T_{31}-T_{51},T_{32}-T_{52})$$
$$I_2=(T_{11}-T_{21}-T_{31}+T_{51},
T_{51}^2-T_{51}T_{31}-T_{21}^2+T_{21}T_{31}-2T_{52}+T_{32},
T_{21}^2-T_{21}T_{51}+T_{52}).$$
We have two maximal lattice Minkowski decompositions of $P$.
$$ 
\raisebox{2ex}{$
\begin{tikzpicture}[scale=0.8]
\draw[thick,  color=black]  
  (0,0) -- (2,0) -- (1,1) -- cycle;
\fill[thick,  color=black]
  (0,0) circle (2.5pt) (1,0) circle (2.5pt) (2,0) circle (2.5pt) (1,1) circle (2.5pt);
\draw[thick,  color=black]
  (2.5, 0.5) node{$\hspace{1.5em}+\hspace{0.8em}$};
\end{tikzpicture}
$}
\raisebox{2ex}{$
\begin{tikzpicture}[scale=0.8]
\draw[thick,  color=black]  
  (0,0) -- (0,1);
\fill[thick,  color=black]
(0,0) circle (2.5pt)  (0,1) circle (2.5pt);
\draw[thick,  color=black]
  (0.7, 0.45) node{$\hspace{1.5em}+\hspace{0.8em}$};
  \end{tikzpicture}
$}
\raisebox{2ex}{$
\begin{tikzpicture}[scale=0.8]
\draw[thick,  color=black]  
  (0,0) -- (0,1);
\fill[thick,  color=black]
(0,0) circle (2.5pt)  (0,1) circle (2.5pt);
\draw[thick,  color=black]
  (0.7, 0.45) node{$\hspace{1.5em}=\hspace{0.8em}$};
  \end{tikzpicture}
$}
\begin{tikzpicture}[scale=0.8]
\draw[thick,  color=black] 
(0,0) -- (2,0);
\draw[thick,  color=black] 
(2,0) -- (2,2);
\draw[thick,  color=black] 
(2,2) -- (1,3);
\draw[thick,  color=black] 
(1,3) -- (0,2);
\draw[thick,  color=black] 
(0,2) -- (0,0);
\fill[thick,  color=black]
  (0,0) circle (2.5pt) (2,0) circle (2.5pt) (2,2) circle (2.5pt) (1,3) circle (2.5pt)
  (0,2) circle (2.5pt) (0,1) circle (2.5pt) (1,0) circle (2.5pt) (2,1) circle (2.5pt);
\draw[thick,  color=black]
  (3, 0.9) node{$\hspace{0.2em}=\hspace{0.5em}$}; 
\end{tikzpicture}
\raisebox{2ex}{$
\begin{tikzpicture}[scale=0.8]
\draw[thick,  color=black]  
  (0,0) -- (0,1) -- (1,0) -- cycle;
\fill[thick,  color=black]
  (0,0) circle (2.5pt) (0,1) circle (2.5pt) (1,0) circle (2.5pt);
\draw[thick,  color=black]
  (1.5, 0.45) node{$\hspace{2.0em}+\hspace{0.8em}$};
\end{tikzpicture}
$}
\raisebox{2ex}{$
\begin{tikzpicture}[scale=0.8]
\draw[thick,  color=black]  
  (0,0) -- (1,1) -- (1,0) -- cycle;
\fill[thick,  color=black]
  (0,0) circle (2.5pt) (1,1) circle (2.5pt) (1,0) circle (2.5pt);
\draw[thick,  color=black]
  (1.5, 0.4) node{$\hspace{2.0em}+\hspace{0.8em}$};
\end{tikzpicture}
$}
\raisebox{2ex}{$
\begin{tikzpicture}[scale=0.8]
\draw[thick,  color=black]  
  (0,0) -- (0,1);
\fill[thick,  color=black]
(0,0) circle (2.5pt)  (0,1) circle (2.5pt);
\draw[thick,  color=black]
  (0.7, 0.45) node{$\hspace{1.5em}\hspace{0.8em}$};
  \end{tikzpicture}
$}
$$
We first consider the Minkowski decomposition on the left.
The map $g^*$ is in this case given by:
$$T_{11}\mapsto Z_0,~~~T_{21}\mapsto Z_0,~~~T_{31}\mapsto Z_1+Z_2,~~~T_{32}\mapsto Z_1Z_2~~~T_{51}\mapsto Z_1+Z_2,~~~T_{52}\mapsto Z_1Z_2.$$
Thus the kernel of $g^*$ equals the ideal $I_1$. The map $g^*$ is for the second Minkowski decomposition given by:
$$T_{11}\mapsto Z_0,~~~T_{21}\mapsto Z_1,~~~T_{31}\mapsto Z_0+Z_2,~~~T_{32}\mapsto Z_0Z_2,~~~T_{51}\mapsto Z_1+Z_2,~~~T_{52}\mapsto Z_1Z_2.$$
Thus the kernel of $g^*$ equals the ideal $I_2$.
The first component (corresponding to $I_1$) is thus isomorphic to $\spec \CC[T_{11},T_{31},T_{32}]$ and the second component (corresponding to $I_2$) is thus isomorphic to $\spec \CC[T_{21},T_{31},T_{51}]$.
\end{example}

\end{document}